\documentclass{article}

\usepackage{geometry}
\usepackage[mathcal]{euscript}
\usepackage{hyperref}
\usepackage{amsmath,amssymb,amsthm}
\usepackage{mathtools}
\usepackage{dutchcal}

\DeclareMathAlphabet{\mathpzc}{OT1}{pzc}{m}{it}
\DeclareMathOperator{\dist}{dist}
\usepackage{xcolor}
\usepackage{cases}
\newtheorem{thm}{Theorem}[section]
\newtheorem{lemma}[thm]{Lemma}
\newtheorem{prop}[thm]{Proposition}
\newtheorem{cor}[thm]{Corollary}
\theoremstyle{remark}
\newtheorem{rem}[thm]{Remark}

\theoremstyle{definition}
\newtheorem{defn}[thm]{Definition}
\newtheoremstyle{Claim}{}{}{\itshape}{}{\itshape\bfseries}{:}{ }{#1}
\theoremstyle{Claim}

\newcommand{\Z}{{\mathbb{Z}}}
\newcommand{\T}{{\mathbb{T}^d}}

\renewcommand{\H}{\mathcal{H}}

\newcommand{\R}{\mathbb{R}}

\newcommand{\N}{\mathbb{N}}
\renewcommand{\L}{\mathcal{L}}

\newcommand\Hb{{}_\beta\H}

\newcommand\Vb{{}_{\beta}\mathcal{V}}

\newcommand{\norm}[1]{\left\lVert#1\right\rVert}

\def\ds{\displaystyle}

\titlepage
\title{Existence and regularity results for viscous Hamilton-Jacobi equations with  Caputo time-fractional derivative}
\author{Fabio Camilli and Alessandro Goffi}
\date{\today}

\begin{document}

\maketitle

\begin{abstract}
We study     existence, uniqueness and regularity properties of classical solutions to viscous   Hamilton-Jacobi equations with Caputo  time-fractional  derivative. Our study relies on a combination of a gradient bound  for the time-fractional Hamilton-Jacobi equation obtained via nonlinear adjoint method  and  sharp estimates in Sobolev and H\"older spaces for the corresponding linear problem. 
\end{abstract}

\noindent
{\footnotesize \textbf{AMS-Subject Classification} 35R11, 26A33, 45K05}.


\tableofcontents

\section{Introduction}
In recent times, to model    memory effects and subdiffusive regimes in  applications   such as transport theory, viscoelasticity, rheology and     nonmarkovian stochastic processes, there has been an increasing interest in the study of time-fractional differential equations, i.e.   differential equations where the standard time derivative is replaced by a fractional one, typically a Caputo or a Riemann-Liouville derivative.  A significant number of papers has been devoted to extend properties holding in the standard setting to the fractional one (see for example  \cite{ACV,DK,LinLiu,Luchko,ZacherGlobal}). \par
Aim of this paper is to study   existence, uniqueness and regularity properties of classical solutions to the time-fractional Hamilton-Jacobi equation
\begin{equation}\label{i;1}
\left\{\begin{array}{ll}
\partial_{(0,t]}^{\beta} u(x,t)-\Delta u+H(x,Du)=V(x,t)\quad&(x,t)\in Q_T=\T\times (0,T),\\
u(x,0)=u_0(x)&x\in\T,
\end{array}
\right.
\end{equation}
where $\T$ is the unit torus, $H$   a convex, coercive Hamiltonian in $Du$ and $\partial_{(0,t]}^{\beta} u$, for $\beta\in(0,1)$,   denotes the Caputo time-fractional derivative
\[
\partial_{(0,t]}^{\beta} u(x,t)=\frac{1}{\Gamma(1-\beta)}\int_0^t\frac{\partial _\tau u(x,\tau)}{(t-\tau)^\beta}\,d\tau.
\]
In the study of \eqref{i;1}, we are mainly motivated by problems arising in Mean Field Games theory for subdiffusion processes (see \cite{cdm}), where typically $H=H(x,p)$ behaves like $|p|^\gamma$, $\gamma>1$ in $p$ and $V$ is a so-called regularizing coupling \cite{CLLP}. \\
Starting with \cite{LSU},  quasi-linear equations of the form \eqref{i;1}  with  standard time derivative have been extensively studied  in literature  and several
results are available  depending on growth conditions on $H$ with respect to the gradient variable (see  \cite{CG2,gomesbook,Porretta} and references therein). Recently, a theory of  weak solutions (in viscosity sense) for the Hamilton-Jacobi equation \eqref{i;1} has been investigated in \cite{giganamba,toppyangari,tly}. In \cite{kolokver},   existence results for \eqref{i;1} are obtained by means of Fourier transform in space. Here, we propose a study  of \eqref{i;1} based on  a combination of   a priori estimates for the linear problem combined  with  the so-called nonlinear adjoint method developed by   Evans (see  \cite{Evans}, \cite{gomesbook} and references therein).  This latter scheme was introduced to study more deeply the gradient shock structures of viscosity solutions to non-convex Hamilton-Jacobi equations. In particular, it relies on studying the adjoint of the linearization of the Hamilton-Jacobi equation via integration by parts rather than relying on maximum principles arguments. In this analysis, since we are dealing with time-fractional derivatives, we introduce the \textit{time-fractional Fokker-Planck equation}
\begin{equation}\label{adjointintro}
\begin{cases}
\partial_{[t,\tau)}^\beta\rho-\sigma\Delta \rho-\mathrm{div}(D_pH(x,Du)\rho)=0&\text{ in }Q_\tau:=\T\times(0,\tau)\ ,\\
\rho(x,\tau)=\rho_\tau(x)&\text{ in }\T,
\end{cases}
\end{equation}
$\tau\in(0,T)$, driven by the drift $b(x,t):=D_pH(x,Du(x,t))$, where
\[
\partial_{[t,\tau)}^\beta\rho=-\frac{1}{\Gamma(1-\beta)}\int_t^\tau\frac{ \partial_s\rho(x,s)}{(s-t)^\beta}\,ds ,
\]
stands for the backward Caputo derivative. We will henceforth refer to the solution $\rho$ of \eqref{adjointintro} as the \textit{adjoint variable}. 
In the first part of the paper, as an important  preliminary  step to the study of \eqref{i;1}, we review and extend maximal regularity results in Lebesgue and Holder spaces for abstract  linear differential equations of the form
\begin{equation}\label{i;2}
\left\{ 
\begin{array}{ll}
\partial_{(0,t]}^{\beta} u(t)+A u(t)=f(t),\quad \text{on $I$},\\
u(0)=u_0,
\end{array}
\right.
\end{equation}
where $I\subseteq\R$, $X$ a Banach space with norm $\|\cdot\|$, $u:I\to X$, $A:D(A)\to X$ is an unbounded linear operator, being $D(A)$ a linear subspace of $X$ (the so-called domain of $A$) equipped with the graph norm $\|x\|_{D(A)}=\|x\|+\|Ax\|$ and $u_0$   belongs to a suitable Banach space. 
From the functional viewpoint, the main peculiarity of \eqref{i;2} is that the usual semigroup property 
fulfilled by the solution operator  is lost because of the memory effect due to the fractional derivative.   
However,   by formally taking the Laplace transform of the equation,    \eqref{i;2} can be rewritten as the following abstract Volterra  equation  
\begin{equation}\label{i;3}
u(t)= u(0)+\int_0^tg_{\beta}(t-\tau) Au(\tau)d\tau+F(t),
 \end{equation}
where $g_{ \beta}(t)=t^{\beta-1}/\Gamma(\beta)$ and $F(t)=\frac{1}{\Gamma(\beta)}\int_0^t(t-\tau)^{\beta-1}f(\tau)d\tau$.
The Volterra equation  \eqref{i;3}  allows to exploit     regularity results   in Lebesgue and H\"older's spaces. In particular, concerning estimates in Lebesgue space,     we recall the classical maximal regularity result  in \cite{ZacherIBVP,Zmax} (see also the monographs \cite{PrussSimonett,PrussBook} and references therein), which is obtained     in the parabolic class
\begin{equation}\label{i;5}
\Vb_p^2(Q):=H_p^\beta(I;L^p (\T))\cap L^p(I;W^{2,p}(\T)).
\end{equation}
Moreover, with the aim of giving a self-contained presentation,  we     provide a detailed discussion of the embeddings for the parabolic spaces $\Vb_p^2(Q)$,    since these results are the basis for the  subsequent study of  \eqref{i;1} via  linearization arguments. We also recall some tools from interpolation theory in Banach spaces, inspired by \cite{KrylovbookSPDE,KrylovJFA} and lately developed in \cite{CG1},  for equations of the form \eqref{i;1} with  standard time derivative. \par
Concerning H\"older's estimates, the analysis of Volterra equations in such framework began with \cite{ClementTAMS,Prussmaximal}.
In the following, we   provide a PDE oriented proof  of these estimates,   which is reminiscent of the classical approach via semigroup  theory to abstract Cauchy problems (see e.g. \cite{LunardiBook} and references therein). For the fractional Laplacian, Schauder's estimates have been investigated  in \cite{CF}, while a different approach, based on interpolation theory methods, has been developed in \cite{CG1}. Here, a crucial ingredient is a Duhamel-like formula for \eqref{i;1} defined via the so-called Mittag-Leffler families.\par
The previous preliminary tools   provide the basis for the analysis of the  time-fractional Hamilton-Jacobi equations. The main result of the paper is the following  existence and uniqueness result for classical solutions to \eqref{i;1} with regularizing right-hand side $V$ (assumptions \eqref{H1}-\eqref{H5} are detailed at the beginning of  Section \ref{sec:existence}).\\
\begin{thm}\label{existenceHamilton-Jacobi}
	Let $\beta\in(0,1)$, $H$ satisfying \eqref{H1}-\eqref{H5}, $p>d+2/\beta$, $u_0\in W_{p}^{2-\frac{2}{p\beta}}(\T)$ and $V\in L^p(Q_T)$. Then, there exists  $\tau^*<T$ such that \eqref{i;1} admits a unique   (strong) solution $u\in \Vb^{2}_p(\T\times(0,\tau^*))$.\\
	Moreover, if $\beta\in(\frac12,1)$, $u_0\in C^{4+\frac{\gamma}{\beta}}(\T)$ and $V\in C([0,T];C^{2+\frac{\gamma}{\beta}}(\T))$, there exists a unique (global) solution $u\in C([0,T];C^{4+\frac{\gamma}{\beta}}(\T))$ to \eqref{i;1}, and the following estimate holds
	\begin{equation*}\label{Lib}
	\norm{u}_{C([0,T];C^{4+\frac{\gamma}{\beta}}(\T))}\leq C(\norm{V}_{C([0,T];C^{2+\frac{\gamma}{\beta}}(\T))}+\norm{u_0}_{C^{4+\frac{\gamma}{\beta}}(\T)})\ .
	\end{equation*}
	\end{thm}
The proof of the previous theorem relies on a  contraction mapping principle, in the spirit of recent   works for time-fractional quasi-linear equations (see \cite{Andrade,ZacherGlobal}). Here, the result is achieved by combining the analysis of the parabolic Sobolev spaces outlined above with an a priori gradient bound on the solution of \eqref{i;1},   which allows to treat \eqref{i;1} as a perturbation of a time-fractional heat equation. For the latter   argument, a crucial step is an estimate of      crossed quantity \begin{equation}\label{i;4}
\iint|Du|^{\gamma}\rho \, dxdt\ ,
\end{equation}
with $\rho$   solving \eqref{adjointintro}, which is accomplished by means of  the aforementioned nonlinear adjoint method. We remark that  bounds of   type  \eqref{i;4} play an important role in Mean Field Games theory and stochastic control \cite{CG1,gomesbook,Porretta}, Sobolev regularity for the Fokker-Planck equation  (\cite{BKRS,MPR,CT}), Hamilton-Jacobi equations  \cite{CG2}. 
Some comments on  space-time H\"older regularity results for \eqref{i;1} as well as on the restriction $\beta\in (1/2,1)$ are provided in Remark \ref{stclassical}.\\

The paper is organized as follows. In Section \ref{sec;notations}, we introduce the functional spaces where the problem is studied and we discuss several of their properties.  Section \ref{sec;heat_eq} is devoted to the study of the time-fractional heat equation \eqref{i;3} and Section \ref{sec;schauder} to the corresponding Schauder estimates. Finally, in Section \ref{sec:existence}, we prove the main results of the paper, existence and  gradient bound on  the solution of \eqref{i;1}.
\bigskip

{\bf Acknowledgements.} The authors are members of the Gruppo Nazionale per l'Analisi Matematica, la Probabilit\`a e le loro Applicazioni (GNAMPA) of the Istituto Nazionale di Alta Matematica (INdAM). The second-named author wishes to thank R. Schnaubelt and R. Zacher for useful discussions and references, the Department SBAI,  Sapienza University of Rome and the Department of Mathematics of the University of Padova for the hospitality during the preparation of the paper.
\section{Notations and preliminaries}\label{sec;notations}
In this section, we shortly review some  functional spaces and their properties. Moreover we prove some  embedding results for fractional Sobolev spaces.
\subsection{H\"older spaces}\label{subsec:holder}
We recall here the definition of H\"older spaces on the torus and then define space-time H\"older spaces typically associated to fractional heat-type equations. Let $\alpha\in(0,1)$ and $k$ be a nonnegative integer. A real-valued function $u$ defined on $\T$ belongs to $C^{k+\alpha}(\T)$ if $u\in C^k(\T)$ and 
\begin{equation*}
[D^{r}u]_{C^\alpha(\T)}:=\sup_{x \neq y\in\T}\frac{|D^ru(x)-D^ru(y)|}{\dist(x,y)^{\alpha}}<\infty,
\end{equation*}
for each multi-index $r$ such that $|r|=k$, where $\dist(x,y)$ is the geodesic distance from $x$ to $y$ on $\T$. Note that in the definition of the previous (and following) seminorm,  $\dist(x,y)$ can be replaced by the euclidean distance $|x-y|$  and the supremum can be taken in $\R^d$ since $u$ can be seen as a periodic function on $\R^d$.
%

%

We first consider some vector-valued H\"older spaces. 
Let $X$ be a Banach space and $\gamma\in(0,1)$. Denote by $C^{\gamma}(I; X)$, $I \subseteq [0,T]$, the space of functions $u: I \rightarrow X$ such that the norm defined as
\begin{equation*}
\norm{u}_{C^{\gamma}(I; X)}:=\sup_{t\in I}\norm{u(t)}_X+\sup_{t \neq \tau}\frac{\norm{u(t)-u(\tau)}_X}{|t-\tau|^\gamma}
\end{equation*}
is finite. Hence,  specializing to $X=C^{\alpha}(\T)$, $\alpha\in(0,1)$, we have that
 $C^{\gamma}(I; C^{\alpha}(\T))$ is the set of functions $u: I \rightarrow C^{\alpha}(\T)$ with finite norm
\begin{equation*}
\norm{u}_{C^{\gamma}(I;C^{\alpha}(\T))}:=\norm{u}_{\infty; Q}+\sup_{t\in I }[u(\cdot,t)]_{C^{\alpha}(\T)}
+[u]_{C^{\gamma}(I;C^{\alpha}(\T))}\ ,
\end{equation*}
where the last seminorm being defined as
\begin{equation*}
[u]_{C^{\gamma}(I;C^{\alpha}(\T))}:=\sup_{t \neq \tau\in I}\frac{\norm{u(\cdot,t)-u(\cdot,\tau)}_{C^{\alpha}(\T)}}{|t-\tau|^{\gamma}}\ .
\end{equation*}
Let now $Q = \T \times I$. We define
\begin{equation*}
[u]_{C^{\alpha}_x(Q)}:=\sup_{t\in I}[u(\cdot,t)]_{C^{\alpha}(\T)}
\end{equation*}
and
\begin{equation*}
[u]_{C^{\gamma}_t(Q)}:=\sup_{x\in\T}[u(x,\cdot)]_{C^{\gamma}(I)}
\end{equation*}

When dealing with regularity of parabolic equations with fractional operators, we also need H\"older spaces with different regularity in time and space. Following the lines of \cite{CF} and \cite{FRRO}, we define $\mathcal{C}^{\alpha,\gamma}(Q)$ as the space of continuous functions $u$ with finite H\"older parabolic seminorm  
\begin{equation}\label{fracholder}
[u]_{\mathcal{C}^{\alpha,\gamma}(Q)}:= [u]_{C^{\alpha}_x(Q)}+[u]_{C^{\gamma}_t(Q)}.
\end{equation}
The norm in the space $\mathcal{C}^{\alpha,\gamma}(Q)$ is defined naturally as
\begin{equation*}
\norm{u}_{\mathcal{C}^{\alpha,\gamma}(Q)}:=\|u\|_{\infty; Q}+[u]_{\mathcal{C}^{\alpha,\gamma}(Q)}\ .
\end{equation*}
Note that if $\gamma = \alpha/2$, the space $\mathcal{C}^{\alpha,\gamma}(Q)$ coincides with $C^{\alpha, \alpha/2}(Q)$. For these latter classical parabolic H\"older spaces, we refer the interested reader to \cite{LSU} for a detailed discussion. As pointed out in \cite{FRRO}, the following equivalence between seminorms holds
\begin{equation*}
[u]_{\mathcal{C}^{\alpha,\gamma}(Q)}\sim \sup_{x,y\in\T,t,\tau\in I}\frac{|u(x,t)-u(y,\tau)|}{\dist(x,y)^{\alpha}+|t-\tau|^{\gamma}} \ .
\end{equation*}

All the spaces above can be defined analogously on $\R^d$ and $ \R^d \times I$. Moreover, if $u$ is a periodic function in the $x$-variable, norms on $\T$ and $\R^d$ coincide, e.g. $\|u\|_{C^\alpha(\T)} = \|u\|_{C^\alpha(\R^d)}$, etc.

\begin{rem}\label{holderinc}
It is worth noticing that we have to distinguish the spaces $C^{\gamma}(I;C^{\alpha}(\T))$ and $\mathcal{C}^{\alpha,\gamma}(Q)$, since it results
\begin{equation*}
C^{\gamma}(I;C^{\alpha}(\T))\subsetneq \mathcal{C}^{\alpha,\gamma}(Q)\ .
\end{equation*}
This can be easily seen by taking $\gamma=\alpha$ and a periodic function in the $x$-variable that behaves like $(x+t)^{\alpha}$ in a neighbourhood of $(0,0)$ (see in particular \cite[Section 4]{Rabier}).
\end{rem}
\subsection{Fractional Sobolev and Bessel potential spaces}\label{susb}
This section is devoted to collect the definitions of Lebesgue and Sobolev spaces we will use throughout the paper. Recall that $L^p(\T)$ is the space of all measurable and periodic functions belonging to $L_{\rm loc}^p(\R^d)$ with norm $\|\cdot\|_{p}=\|\cdot\|_{L^p((0,1)^d)}$. If $k$ is a nonnegative integer, $W^{k, p}(\T)$ consists of $L^p(\T)$ functions with (distributional) derivatives in $L^p(\T)$ up to order $k$. For $\mu\in\R$ and $p\in(1,\infty)$, we can define the Bessel potential space $H_p^{\mu}(\T)$ as the space of all distributions $u$ such that $(I-\Delta)^{\frac{\mu}{2}}u\in L^p(\T)$, where $(I-\Delta)^{\frac{\mu}{2}}u$ is the operator defined in terms of Fourier series
\begin{equation*}
(I-\Delta)^{\frac{\mu}{2}} u(x)=\sum_{k\in\Z^d}(1+4\pi^2|k|^2)^{\frac{\mu}{2}} \hat u(k) e^{2\pi ik\cdot x}\ ,
\end{equation*}
and
\begin{equation*}
\hat u(k)=\int_{\T}u(x)e^{-2\pi ik\cdot x}dx\ .
\end{equation*}
The norm in $H_p^{\mu}(\T)$ will be denoted by 
\begin{equation*}
\norm{u}_{\mu,p}:=\norm{(I-\Delta)^{\frac{\mu}{2}}u}_{p}.
\end{equation*}
Note that $H_p^{\mu}(\T)$ coincides with $W^{\mu, p}(\T)$ when $\mu$ is a nonnegative integer and $p\in(1,\infty)$ (see \cite[Remark 2.3]{CG1}).

We recall that Bessel potential spaces can be also constructed via complex interpolation (for additional details, we refer to \cite[Chapter 2]{LunardiSNS}, \cite{CG1} and references therein). More precisely, for $\mu\in\R$, the space $H_p^{\mu}(\T)$ can be obtained by complex interpolation between $L^p(\T)$ and $W^{k,p}(\T)$ as
\[
H^\mu_p(\T) \simeq [L^p(\T),W^{k,p}(\T)]_{\theta}, \qquad \text{where $\mu = k\theta$}.
\]
Bessel potential spaces can be defined also in $\R^d$ in the same manner. We further recall that the operator $(I-\Delta)^{\frac{\mu}{2}}$ maps isometrically $H_p^{\eta+\mu}(\T)$ to $H_p^{\eta}(\T)$ for any $\eta,\mu\in\R$ (see \cite[Remark 2.3]{CG1}).

We shortly describe the so-called K-method, which   allows to define fractional Sobolev spaces $W^{\mu,p}$ as ``intermediate" spaces between $L^p(\T)$ and $W^{k,p}(\T)$, but also H\"older's spaces when $p=\infty$. Let $X,Y$ be Banach spaces with $Y\subset X$, $\theta\in[0,1]$ and $p\in[1,\infty]$. For every $x\in X$ and $t>0$, define
\[
K(t,x,X,Y)=\inf_{x=a+b,a\in X,b\in Y}\|a\|_X+t\|b\|_Y\ .
\]
If $I\subset(0,\infty)$, we denote by $L^p_{*}(I)$ the Lebesgue space $L^p(I,\frac{dt}{t})$ and $L^{\infty}_{*}(I)=L^{\infty}(I)$. We define the real interpolation space $(X,Y)_{\theta,p}$ between the Banach spaces $X,Y$ as
\[
(X,Y)_{\theta,p}=\{x\in X+Y:t\mapsto t^{-\theta}K(t,x,X,Y)\in L^p_{*}(0,+\infty)\},
\]
endowed with the norm
\[
\|x\|_{\theta,p}=\|t^{-\theta}K(t,x,X,Y)\|_{L^p_{*}(0,+\infty)}\ .
\]
It can be proved that $(X,Y)_{\theta,p}$ is a Banach space. Then, one shows (see e.g. \cite[Example 1.8]{LunardiSNS}) that
\[
(C(\T),C^1(\T))_{\theta,\infty}=C^\theta(\T)
\]
and
\[
(L^p(\T),W^{1,p}(\T))_{\theta,p}=W^{\theta,p}(\T)\ .
\]
We remark that such tool turns out to be useful to prove H\"older regularity of the solution of the time-fractional heat equation in Theorem \ref{Regularity}. We finally point out that the aforementioned isometry properties via $(I-\Delta)^{\frac{\mu}{2}}$ transfer also to the fractional Sobolev spaces $W^{\mu,p}$ (see e.g. \cite[Theorem 6.2.7]{BL}).
\subsection{Parabolic Sobolev spaces}\label{sec;parab_sob}
In the next section we investigate some properties of the space 
\[
\Vb_p^2(Q):=H_p^\beta(I;L^p(\T))\cap L^p(I;W^{2,p}(\T)),
\]
which is the suited one to deal with $L^p$-maximal regularity for time-fractional PDEs (see \cite{PrussBook,PrussSimonett,ZacherGlobal,ZacherIBVP}). 
 Here, the vector-valued space $H_p^\beta(I;X)$, $X$ being a Banach space and $I$ an open subset of the real line, can be defined via the aforementioned complex interpolation as
\begin{equation}\label{besselone}
H_p^\beta(I;X)\simeq [L^p(I;X);W^{1,p}(I;X)]_{\beta}
\end{equation}
for $\beta\in(0,1)$. Moreover, the vector-valued Slobodeckij scale $W^{\beta,p}(I;X)$, $\beta\in\R^+\backslash\N$ consists of all functions $u\in W^{[\beta],p}(I;X)$ such that $[D^\alpha u]_{\beta-[\beta],p}<\infty$ for $\alpha=[\beta]$, where $[\beta]$ is the integer part of $\beta$, being
\[
[u]_{\theta,p}^p:=\int_{I\times I}\frac{\|u(\omega)-u(\eta))\|_{X}^p}{|\omega-\eta|^{d+\theta p}}d\theta d\eta\ , \theta\in(0,1)\ .
\]
Since one usually needs estimates in $I\subseteq[0,T]$ stable at some time, it is necessary to have a control of the trace, e.g. on the hyperplane $t=0$. In the classical parabolic case, the initial trace can be characterized via the trace method in interpolation theory in Banach space (see e.g. \cite[Section 1.2]{LunardiSNS} and the discussion in \cite[Remark B.4]{CG1}, see also \cite[Section 3.4]{PrussSimonett}). For instance, in the case of $W^{2,1}_p(\T\times I)$ the initial trace $u(0)$ turns out to belong to the fractional Sobolev space $W^{2-2/p,p}(\T)\simeq (L^p(\T),W^{2,p}(\T))_{1-1/p,p}$ (see e.g. \cite[Lemma II.3.4]{LSU} and \cite[Corollary 1.14]{LunardiSNS}). In the time-fractional case,  it is useful to transform the time-fractional PDE into an abstract Volterra equations of the form
 \eqref{i;3} and use the resolvent approach, which can be seen as a generalization of the classical   semigroup analysis to abstract Cauchy problems. Indeed, by classical results (see e.g. \cite[Theorem 4.2]{PrussBook}) we know that by the subordination principle $A=-\Delta$ admits a resolvent. Exploiting properties of Laplace transform and the resolvent family of the problem (see e.g. the discussion in \cite[Section 5.4]{PrussSimonett}) it turns out that the space of initial traces of functions in $\Vb_p^2$ is the fractional space $W^{2-2/(p\beta),p}$  for $\beta>1/p$ (see \cite[Proposition 4.5.14]{PrussSimonett} and also \cite{MV})

Recently,   time-fractional PDEs with null initial trace have been investigated in the context of the parabolic spaces 
\begin{equation*}
\Hb_p^{\mu}(Q) := \{u \in L^p(I;H^\mu_p(\T)), \, \partial_{(0,t]}^{\beta} u \in  L^p(I;H_p^{\mu-2}(\T)) \} 
\end{equation*}
for $\mu=2$,   $p>1$ (see e.g. \cite{DK,KKL}).
The previous Sobolev spaces are clearly reminiscent of the  parabolic spaces $W^{2,1}_p$ typically associated to the heat operator $\partial_t-\Delta$  (i.e. when $\mu=2$ and $\beta=1$).
We note that the space $\Vb_p^2$ is isomorphic to the parabolic space $\Hb_p^2$ in the case of zero initial trace $u(0)=0$ (see e.g. \cite[Proposition 2]{CLS}). Indeed, this can be seen  by the representation in \eqref{besselone} setting $X=W^{2,p}(\T)$
and exploiting that, for sufficiently smooth $u$, $\partial_{(0,t]}^{\beta} u=\partial_t(g_{1-\beta}\star u)$, which in turn allows to write
\begin{multline*}
\{u\in L^p(I;W^{2,p}(\T)): \partial_{(0,t]}^{\beta} u\in L^p(Q)\ ,u(0)=0\}\\
\simeq \{ u\in L^p(I;W^{2,p}(\T)): \partial_t(g_{1-\beta}\star u)\in L^p(Q)\ ,u(0)=0\}\\
\simeq \{ u\in L^p(I;W^{2,p}(\T)): g_{1-\beta}\star u\in W^{1,p}(0,T;L^p(\T))\ ,u(0)=0\}
\end{multline*}
and concluding using \cite[Section 1.15.3]{trbookinterpolation}.

In Section \ref{sec;embedd}, we will also cover the case of parabolic spaces frequently associated to the fully nonlocal operator 
\begin{equation}\label{space_time_operator}
\partial_{(0,t]}^{\beta}+(-\Delta)^s\ ,
\end{equation}
where $(-\Delta)^s$ denotes the standard fractional Laplacian (see e.g. \cite{CG1} and references therein for a treatment on the torus). To this aim, we denote by
\[
\Vb^{2s}_p(Q):=H_p^\beta(I;L^p(\T))\cap L^p(I;H^{2s}_p(\T))\ .
\]
Such spaces were first investigated in the context of stochastic PDEs in \cite{CL}. In \cite{CG1}, they are studied in connection with space-fractional   PDEs   using  arguments inspired by \cite{KrylovbookSPDE,KrylovJFA}. These spaces are natural in the context of parabolic PDEs and the corresponding  $L^p$   theory is crucial in the  study of parabolic regularity properties even for equations with divergence-type terms (see \cite{BKRS,CG2,CT,MPR,Porretta}). 


\subsection{Embedding results for parabolic Sobolev spaces.}\label{sec;embedd}
In this section, we prove some embedding results for the fractional parabolic spaces introduced in Section \ref{sec;parab_sob} which will be applied in the study of \eqref{i;1}, though they are of independent interest.
We first recall some classical embeddings for fractional Sobolev spaces $W^{\mu,p}$ and $H_p^\mu$, $\mu\in\R$.
\begin{lemma}\label{inclstatW}
	\begin{itemize}
		\item[(i)] Let $\nu,\mu\in\R$ with $\nu\leq\mu$, then $W^{\mu,p}(\T)\subset W^{\nu,p}(\T)$.
		\item[(ii)] If $p\mu>d$ and $\mu-d/p$ is not an integer, then $W^{\mu,p}(\T)\subset C^{\mu-d/p}(\T)$.
		\item[(iii)] Let $\nu,\mu\in\R$ with $\nu\leq\mu$, $p,q\in(1,\infty)$ and
		\[
		\mu-\frac{d}{p}=\nu-\frac{d}{q}\ ,
		\]
		then $W^{\mu,p}(\T)\subset W^{\nu,q}(\T)$.
	\end{itemize}
\end{lemma}
\begin{proof}
	These results are well known in $\R^d$
	(see \cite[Section 2.8.1]{trbookinterpolation}). The transfer  to the periodic setting can be obtained exactly as in \cite[Lemma 2.5]{CG1}.
\end{proof}
For the proof of following result we refer to  \cite[Lemma 2.5]{CG1}.
\begin{lemma}\label{inclstatH}
	\begin{itemize}
		\item[(i)] Let $\nu,\mu\in\R$ with $\nu\leq\mu$, then $H^{\mu}_p(\T)\subset H^{\nu}_p(\T)$.
		\item[(ii)] If $p\mu>d$ and $\mu-d/p$ is not an integer, then $H^{\mu}_p(\T)\subset C^{\mu-d/p}(\T)$.
		\item[(iii)] Let $\nu,\mu\in\R$ with $\nu\leq\mu$, $p,q\in(1,\infty)$ and
		$$\mu-\frac{d}{p}=\nu-\frac{d}{q},$$  then $H^{\mu}_p(\T)\subset H^{\nu}_q(\T)$.
	\end{itemize}
\end{lemma}
\begin{rem}\label{vector-valued}
Similar embeddings continue to hold for the vector-valued spaces $W_p^\beta(I;X)$ and $H_p^\beta(I;X)$, being mainly based on interpolation theory arguments. We refer the reader to \cite[Proposition 2.10]{MS} and references therein for further details.
\end{rem}
\begin{rem}\label{ExTime}
By standard extension arguments with respect to the time variable and multiplication by cut-off functions one, can extend functions on $\Vb_p^2(Q_T)$ to $\Vb_p^2(\T\times\R^+)$ and apply the related embedding results for the half-line case (see e.g. \cite[Lemma 2.5]{MS})
\end{rem}
We can now state a Sobolev embedding theorem for the parabolic space $\Vb_p^2$.
\begin{thm}\label{holder}
	Let $p>1$, $u\in \Vb_p^{2}(Q_T)$ and $u(0)\in W^{2-\frac{2}{p\beta},p}(\T)$. If $\alpha$ is such that
	\[
	\frac{1}{p\beta}<\alpha<1,
	\]
	then $\Vb_p^{2}(Q)$ is continuously embedded onto $C^{\alpha\beta-\frac1p}([0,T];H_p^{2-2\alpha}(\T))$ and there exists a constant $C>0$ such that
	\[
	\|u\|_{C^{\alpha\beta-\frac1p}([0,T];H_p^{2-2\alpha}(\T))}\leq C(\|u\|_{\Vb_p^{2}(Q_T)}+\|u(0)\|_{W^{2-\frac{2}{p\beta},p}(\T)}).
	\]
	Note that the constant $C=C(d,p,\alpha,\beta,T)$ remains bounded for bounded values of $T$.
\end{thm}
\begin{proof}
	Let first $u(0)=0$. This fact is basically a consequence of the mixed derivative theorem (see \cite[Corollary 4.5.10]{PrussSimonett} and \cite[Proposition 2.3.2 and Chapter 3]{TesiZacher} for further discussions) after using Remark \ref{ExTime}, which allows to obtain for $r\in[0,1]$ the embedding
	\[
	\Vb_p^{2}(Q_T)\hookrightarrow 
	H_p^{r}(0,T;H_p^{2-2r/\beta}(\T)).
	\]
	We take $r:=\alpha\beta\in (0,1)$ and this gives
	\[
	\Vb_p^{2}(Q_T)\hookrightarrow H_p^{\alpha\beta}(0,T;H_p^{2-2\alpha}(\T)).
	\]
	By using Remark \ref{vector-valued}, we get the embedding
	\[
	H_p^{\alpha\beta}(0,T;H_p^{2-2\alpha}(\T))\hookrightarrow C^{\alpha\beta-1/p}([0,T];H_p^{2-2\alpha}(\T)).
	\]
	
	To show the fact that the embedding constant can be bounded independently of $T>0$ by adding $\|u(0)\|_{W^{2-\frac{2}{p\beta},p}(\T)}$ one can argue as follows. After extending $u$ to $\hat u$ on $\R^+$ as pointed out in Remark \ref{ExTime}, it suffices to subtract a function $\bar{u}\in \Vb_p^2(\T\times\R^+)$ such that $\bar{u}(0)=u(0)$ and $\|\bar{u}\|_{\Vb_p^2(\T\times\R^+)}\leq C\|u(0)\|_{W^{2-\frac{2}{p\beta},p}(\T)}$. Then one concludes by applying the results for null initial traces to $u-\bar{u}$.

\end{proof}

We will need the following embedding onto H\"older classes.
\begin{prop}\label{emb}
	The space $\Vb_p^2(Q_T)$ is continuously embedded onto $C([0,T];W^{2-\frac{2}{p\beta},p}(\T))$. Moreover, the space $C([0,T];W^{1-\frac{2}{p\beta},p}(\T))$ is continuously embedded onto $C([0,T];C^{\gamma/\beta}(\T))$ for some $\gamma\in(0,1)$ when $p>d+2/\beta$.
\end{prop}
\begin{proof}
	The fact that
	\[
	\Vb_p^{2}(Q_T)\hookrightarrow C([0,T];W^{2-\frac{2}{p\beta},p}(\T))
	\]
	is a proven in \cite[Theorem 4.5.15]{PrussSimonett}. As for the second assertion, by exploiting classical embedding theorems for fractional Sobolev spaces (see Lemma \ref{inclstatW}-(ii)), we get that
	\[
	C([0,T];W^{1-\frac{2}{p\beta},p}(\T))\hookrightarrow C([0,T];C^{\frac{\gamma}{\beta}}(\T)),
	\]
	whenever
	\[
	\left(1-\frac{2}{p\beta}\right)p>d,
	\]
	namely  for $p>d+2/\beta$ and   $\gamma=\beta-2/p-d\beta/p\in(0,1)$.
\end{proof}

\begin{prop}\label{lebesgue}
	Let $q\geq p>1$, $\theta\in\R$ such that
	\[
	\eta<2+\frac{d}{q}-\frac{d+\frac{2}{\beta}(1-\theta)}{p}.
	\]
	Then, for any $u\in\Vb_p^{2}(Q_T)$, we have
	\[
	\left(\int_0^T\|u(\cdot,t)\|_{\eta,q}^{\frac{p}{\theta}}dt\right)^{\theta}\leq C(\|u\|_{\Vb_p^{2}(Q_T)}^p+\|u(0)\|_{W^{2-\frac{2}{p\beta},p}(\T)}^p).
	\]
	Moreover, if $1<p<d+\frac{2}{\beta}$ and $\frac1q>\frac1p-\frac{2}{d+\frac{2}{\beta}}$
	\[
	\|u\|_{L^q(Q_T)}\leq C(\|u\|_{\Vb_p^{2}(Q_T)}+\|u(0)\|_{W^{2-\frac{2}{p\beta},p}(\T)}).
	\]
\end{prop}
\begin{proof}
	Let $\nu=(2-2\alpha)(1-\theta)+2\theta$, $\alpha>1/(p\beta)$. Then, we note that $H_p^\nu$ can be obtained as real interpolation between $H_p^2$ and $H_p^{2-2\alpha}$. Moreover, $H_p^\nu$ is continuously embedded onto $H_q^{\nu+d/q-d/p}$ in view of Lemma \ref{inclstatH}-(iii).
	Hence, for a.e. $t$, we have
	\[
	C(d,p,s,\beta,\alpha)\|u(t)\|_{\nu+d/q-d/p}\leq \|u(t)\|_{\nu}\leq \|u(t)\|_{2-2\alpha,p}^{1-\theta}\|u(t)\|_{2,p}^\theta\ .
	\]
	Owing to the inequality $\alpha>\frac{1}{\beta p}$ we then obtain
	\begin{multline*}
	\eta\leq \nu-\frac{d}{p}+\frac{d}{q}=(2-2\alpha)(1-\theta)+2\theta-\frac{d}{p}+\frac{d}{q}\\
	=2-2\alpha(1-\theta)-\frac{d}{p}+\frac{d}{q}<2+\frac{d}{q}-\frac{d+\frac{2}{\beta}(1-\theta)}{p}.
	\end{multline*}
	Denote by $\mathbb{H}_p^2(Q_T)$   the space $L^p(0,T;H_p^2(\T))$. Then
	\begin{equation}\label{embeddingsottocrit}
	\begin{split}
	\left(\int_0^T\|u(\cdot,t)\|_{\eta,q}^{\frac{p}{\theta}}dt\right)^{\theta}&\leq\left(\int_0^T\|u(t)\|_{2-2\alpha,p}^{(1-\theta)\frac{p}{\theta}}\|u(t)\|_{\mu,p}^{p}dt\right)^\theta\\
	&\leq C\sup_{t\leq T}\|u(t)\|_{2-2\alpha,p}^{(1-\theta)p}\left(\int_0^T\|u(t)\|_{2,p}^{p}dt\right)^\theta\\
	&\leq C(\|u\|_{\Vb_p^{2}(Q_T)}+\|u(0)\|_{W^{2-\frac{2}{p\beta},p}(\T)})^{(1-\theta)p}\|u(t)\|_{\mathbb{H}_p^{2}(Q_T)}^{\theta p}\\
	&\leq C(\|u\|_{\Vb_p^{2}(Q_T)}+\|u(0)\|_{W^{2-\frac{2}{p\beta},p}(\T)})^p,	
	\end{split}
	\end{equation}
	where in the last inequality we used Theorem \ref{holder} and Young's inequality. The last statement can be obtained by setting $\eta=0$ and $\theta=\frac{p}{q}$. 
\end{proof}
\begin{rem}
	We can actually reach the the threshold \[
	\frac1q=\frac1p-\frac{2}{d+\frac{2}{\beta}}\]
	using a maximal regularity result. This is accomplished by using the embedding presented in Proposition \ref{emb}, namely  
	\[
	\Vb_p^{2}(Q_T)\hookrightarrow C([0,T];W^{2-\frac{2}{p\beta},p}(\T)),
	\]
	instead of Theorem \ref{holder},  writing $H_p^\nu$   as the real interpolation between $H_p^{2-2/(p\beta)}$ and $H_p^2$ via the very same procedure.\\
	 We  also remark that the above embedding theorems can be compared (and are consistent) with the classical ones for the spaces $W^{2,1}_p$ (i.e. corresponding to the case $\beta=1$) presented in \cite[Lemma II.3.3]{LSU} (see \cite{GopalaRaoTAMS,Bagby} for a proof).
\end{rem}
We conclude the embedding  results by providing a Sobolev embedding theorem onto the parabolic H\"older spaces introduced in Section \ref{subsec:holder} (compare with the classical embeddings for the space $W^{2,1}_p$ onto H\"older's classes in \cite[Lemma II.3.3]{LSU}).
\begin{cor} 
	Let $p>\frac{d}{2}+\frac{1}{\beta}$. Then $\Vb_p^2(Q_T)$ is continuously embedded onto $\mathcal{C}^{\frac{\gamma}{\beta},\frac{\gamma}{2}}(Q_T)$ for
	\[
	\gamma=\beta-\frac1p-\frac{d\beta}{2p}\in(0,1).
	\]
	\end{cor}
\begin{proof}
This is a consequence of  Theorem \ref{holder} by taking 
\[
\alpha\beta-\frac1p=\frac{\gamma}{2}
\]
and exploiting the embedding, see   Lemma \ref{inclstatH}-(ii), of $H_p^{2-2\alpha}$ onto $C^{2-2\alpha-d/p}$. By the above choice of $\gamma$, one immediately checks that
\[
2-2\alpha-\frac{d}{p}=\frac{\gamma}{\beta}.
\]
\end{proof}
\begin{rem}\label{spacetimeholder}
Compare the above result with \cite[Corollary 7.19]{DK}. Similar strategies to those in Theorem \ref{holder} can be implemented to show that the space $H_p^{\beta/2}(L^p)\cap L^p(W^{1,p})$ is embedded onto $C^\frac{\delta}{2}([0,T];C^\frac{\delta}{\beta}(\T))\subset\mathcal{C}^{\frac{\delta}{\beta},\frac{\delta}{2}}(Q_T)$ for $\delta=\beta/2-1/p-d\beta/(2p)$ when $p>d+2/\beta$. This will be useful to study regularity in space-time H\"older spaces of quasilinear problems with time-fractional Caputo derivative via a linearization procedure,  see Remark \ref{stclassical} below.
\end{rem}
The arguments previously discussed can be easily extended to yield  parabolic embedding results for space-time fractional spaces associated to the operator \eqref{space_time_operator}.
We state the results without giving the proof, being similar to the above case. See \cite[Section 2]{CG1} to compare the embeddings with those corresponding to the case $\beta=1$ and $s\in(0,1)$.
\begin{thm}\label{holder_fractional}
	Let $p>1$, $u\in \Vb_p^{2s}(Q_T)$ and $u(0)\in W^{2s-\frac{2s}{p\beta},p}(\T)$. If $\alpha$ is such that
	\[
	\frac{s}{p\beta}<\alpha<1,
	\]
	then $\Vb_p^{2s}(Q_T)$ is continuously embedded onto $C^{\frac{\alpha\beta}{s}-\frac1p}([0,T];H_p^{2s-2\alpha}(\T))$ and there exists a constant $C>0$ such that
	\[
	\|u\|_{C^{\frac{\alpha\beta}{s}}([0,T];H_p^{2s-2\alpha}(\T))}\leq C(\|u\|_{\Vb_p^{2s}(Q_T)}+\|u(0)\|_{W^{2s-\frac{2s}{p\beta},p}(\T)}).
	\]
	Note that the constant $C=C(d,p,\alpha,\beta,s,T)$ remains bounded for bounded values of $T$.
\end{thm}
\begin{prop}\label{lebesgue2}
Let $q\geq p>1$, $\theta$, $s\in(0,1)$ such that
\[
\eta<\mu+\frac{d}{q}-\frac{d+\frac{2s}{\beta}(1-\theta)}{p}.
\]
Then, for any $u\in\Vb_p^{2s}(Q_T)$ we have
\[
\left(\int_0^T\|u(\cdot,t)\|_{\eta,q}^{\frac{p}{\theta}}dt\right)^{\theta}\leq C(\|u\|_{\Vb_p^{2s}(Q_T)}^p+\|u(0)\|_{W^{2s-\frac{2s}{p\beta},p}(\T)}^p).
\]
Moreover, if $\mu>0$ and $1<p<d+\frac{2s}{\beta}$ and $\frac1q>\frac1p-\frac{2s}{d+\frac{2s}{\beta}}$
\[
\|u\|_{L^q(Q_T)}\leq C(\|u\|_{\Vb_p^{2s}(Q_T)}+\|u(0)\|_{W^{2s-\frac{2s}{p\beta},p}(\T)}).
\]
\end{prop}
\section{On time-fractional heat equations: $L^p$-maximal regularity results and representation of solutions}\label{sec;heat_eq}
In this section, we collect some definitions and results for the abstract linear problem \eqref{i;3}. 
We first recall the following notion of solution for \eqref{i;2} (see \cite{PrussBook}).
\begin{defn}\label{strong}
If $f\in L^p(I;X)$, a function $u\in L^p(I;X)$  is said to be a \textit{strong  solution} of \eqref{i;2} on $I$ if $u\in L^p(I;D(A))$ and \eqref{i;3} holds  a.e. on $I$.
\end{defn}


Throughout this paper we will mainly work with strong  solutions belonging to the parabolic space $\Vb_p^2$. Anyhow, we will not specify during our bootstrap procedure which kind of solution we mean, being implicit 
from the context.  At the end, in the existence theorem for the time-fractional Hamilton-Jacobi equations, we will show the existence of  a classical solution of the problem.
\begin{prop}\label{equivalence}
Let $u$ be a strong solution to \eqref{i;2} for $A=-\Delta$ and $f\in C(I;X)$, $X$ being a Banach space. Then   \eqref{i;2} can be rewritten as the Volterra equation \eqref{i;3}.
\end{prop}
\begin{proof}
The result is well-known and it can be seen by applying the Riemann-Liouville integral operator to both sides of \eqref{i;2} (see e.g. \cite{PrussBook}, \cite{Vergara} and \cite{PrussSimonett}). Furthermore, it can be also obtained via a Laplace transform approach (see e.g. \cite[Appendix T]{AV}).
\end{proof}
The previous proposition shows that maximal regularity result  to equation \eqref{i;2} can be inferred from that of \eqref{i;3}. We thus have the following important result, whose proof   can be found in classical references for abstract Volterra equation (see e.g.  the recent survey \cite{Zmax}), \cite[Theorem 4.5.15]{PrussSimonett} and also \cite{PrussBook,ZacherIBVP}).
\begin{thm}\label{thm;maxreg}
Let $p>1$, $\beta\in(0,1)$ be such that $\beta>\frac{1}{p}$, $f\in L^p(Q)$ and $u_0\in W^{2-\frac{2}{p\beta},p}(\T)$. Then there exists a unique strong solution $u\in \Vb_p^2(Q)$ to \eqref{i;2} if and only if $f\in L^p(Q_T)$ and $u_0\in W^{2-\frac{2}{p\beta},p}(\T)$ and it holds the estimate
\begin{equation}\label{estczV}
\|u\|_{\Vb_p^{2}(Q)}\leq C(\|f\|_{L^p(Q)}+\|u_0\|_{W^{2-\frac{2}{p\beta},p}(\T)}).
\end{equation}
\end{thm}
We also point out that solutions to \eqref{i;2} can be expressed via a suitable modification of the classical variation of parameters formula, also known as \textit{Duhamel's formula}. A proof is provided in \cite{TaylorNote} in the Hilbert space setting. However, the proof can be readily accommodated to handle   Banach spaces. 
For each $\beta\in(0,1)$ we define the Mittag-Leffler operators
\begin{align*}
&E_\beta(-t^\beta A)=\int_0^\infty M_\beta(\eta)e^{-\eta t^\beta A}d\eta,
\\
&E_{\beta,\beta}(-t^\beta A)=\int_0^\infty \beta \eta M_\beta(\eta)e^{-\eta t^\beta A}d\eta,
\end{align*}
where 
\[M_\beta(z)=\sum_{n=0}^\infty\frac{z^n}{n!\Gamma(1-\beta(1+n))}\] 
denotes    the Mainardi function (see \cite{TaylorNote}). We recall the following useful property of $M_\beta$ (see \cite[Proposition 2]{Neto}).
\begin{prop}\label{propM}
For $\beta\in(0,1)$ and $-1<r<\infty$, when restricting $M_\beta$ to the positive real line, it holds
\[
M_\beta(t)\geq0\text{ for all $t\geq0$, }\quad\int_0^\infty t^rM_\beta(t)dt=\frac{\Gamma(r+1)}{\Gamma(\beta r+1)}.
\]
\end{prop}
Note that $E_\beta(-t^\beta A)$ and $E_{\beta,\beta}(-t^\beta A)$ are well-defined from $X$ into $X$ and, for every $x\in X$, the functions $t\longmapsto E_{\beta}(-t^\beta A)x$ and $t\longmapsto E_{\beta,\beta}(-t^\beta A)x$ are analytic from $[0,\infty)$ to $X$. The Mittag-Leffler operators do not generate semigroups, but they fulfill some additional properties which are close to that of semigroups. For instance, the function $t\longmapsto E_{\beta}(-t^\beta A)x$ is continuous and analytical when $e^{tA}$ generate an analytic semigroup and satisfies
\[
\partial_{(0,t]}^{\beta} E_\beta(-t^\beta A)x=-AE_\beta(-t^\beta A)x\ ,t>0.
\]
\begin{prop}
Let $u$ be a strong solution to  \eqref{i;2}. Then
\begin{equation}\label{duhamel}
u(t)=E_{\beta}(-t^\beta A)u_0+\int_0^\infty \omega^{\beta-1}E_{\beta,\beta}(-\omega^\beta A)f(t-\omega)d\omega.
\end{equation} 
\end{prop}
\begin{proof}
This result is classical and can be found in e.g. \cite{KellerSegel,TaylorNote,WCX}. It can be deduced using the following identity 
\begin{equation}\label{identity}
\L(\partial_{(0,t]}^{\beta} u)(\omega)=\omega^\beta\L u(\omega)-\omega^{\beta-1}u(0),
\end{equation}
for the Caputo derivative (see e.g. \cite[Proposition 3.13]{LiLiuODE}), where $\L$ stands for the Laplace transform, together with the expression of the Laplace transform of $E_\beta(z)$  (see  \cite[eq. (8.7)]{TaylorNote}). 
\end{proof}
\begin{rem}
A solution of \eqref{i;2}
given by  formula \eqref{duhamel} is usually referred as a mild solution of the problem (see e.g. \cite{KellerSegel}). We remark that if $u$ is a strong solution to an abstract Volterra equation, then $u$ is also a mild solution according to \cite[Definition 1.1]{PrussBook} and, in particular, the variation of parameter formula holds as a consequence of \cite[Proposition 1.2 and 1.3]{PrussBook}.

We also remark that a function described by the variation of parameter formula is not necessarily a strong solution (see e.g. \cite[Proposition 1.2 and Proposition 1.3]{PrussBook}
 and also \cite[Section 3]{Vergara} for further discussions).
\end{rem}

\section{Schauder estimates for the time-fractional heat equation}\label{sec;schauder}
This section is devoted to collect some Schauder-type results for problem \eqref{i;2}.
We first present without proof  an abstract result, giving necessary and sufficient condition for maximal H\"older's regularity for the problem \eqref{i;2}, referring to \cite{ClementTAMS} and the recent results in \cite{Guidetti}.
\begin{thm}\label{Regularity}
Let $I\subseteq\R$ be closed. We have the following:
\begin{itemize}
\item[(i)] Let $\theta\in(0,1)$. There exists a unique classical solution $u$ such that $\partial_{(0,t]}^{\beta} u$ and $Au$ are bounded with values in $(X,D(A))_{\theta,\infty}$ if and only if
\begin{itemize}
\item[(a)] $f\in C(I;X)\cap B(I;(X,D(A))_{\theta,\infty})$, where $B(I;X)$ stands for the space of bounded functions with values in $X$;
\item[(b)] $u_0\in D(A)$ and $Au_0\in (X,D(A))_{\theta,\infty}$.
\end{itemize}
and it holds the estimate
\begin{equation}\label{heat;1}
\|\partial_{(0,t]}^{\beta} u\|_{B(I;(X,D(A))_{\theta,\infty})}+\|A u\|_{B(I;(X,D(A))_{\theta,\infty})}\leq C(\|A u_0\|_{(X,D(A))_{\theta,\infty}}+\|f\|_{B([0,T];(X,D(A))_{\theta,\infty})})
\end{equation}

\item[(ii)]Let $\gamma<\beta$ and $\gamma,\beta\in(0,1)$. Then there exists a unique classical solution to \eqref{i;2} such that $\partial_{(0,t]}^{\beta} u,\, Au$ belong to $C^{\frac{\gamma}{2}}(I;X)$ if and only if
\begin{itemize}
\item[(c)] $f\in C^{\frac{\gamma}{2}}(I;X)$;
\item[(d)] $u_0\in D(A)$ and $Au_0+f(0)\in (X,D(A))_{\gamma/(2\beta),\infty}$.
\end{itemize}
\end{itemize}
In particular, it holds the estimate 
\begin{equation}\label{heat;2}
\|\partial_{(0,t]}^{\beta} u\|_{C^{\frac{\gamma}{2}}(I;X)}+\|A u\|_{C^{\frac{\gamma}{2}}(I;X)}\leq C(\|A u_0\|_{(X,D(A))_{\gamma/(2\beta),\infty}}+\|f\|_{C^{\frac{\gamma}{2}}(I;X)})
\end{equation}
\end{thm}
As a byproduct, from estimates \eqref{heat;1} and \eqref{heat;2}, by taking $X=C(\T)$, $\theta=\gamma/(2\beta)$, $A$ as the realization of the Laplacian in $X$ so that $(X,D(A))_{\gamma/(2\beta),\infty}\simeq C^{\frac{\gamma}{\beta}}(\T)$, we get the H\"older  estimate
\begin{equation}\label{heat;3}
\|\partial_{(0,t]}^{\beta} u\|_{\mathcal{C}^{\frac{\gamma}{\beta},\frac{\gamma}{2}}(Q)}+\|-\Delta u\|_{\mathcal{C}^{\frac{\gamma}{\beta},\frac{\gamma}{2}}(Q)}\leq C(\|f\|_{C^\frac{\gamma}{2}(I;C^{\frac{\gamma}{\beta}}(\T))}+\|u_0\|_{C^{2+\frac{\gamma}{\beta}}(\T)}).
\end{equation}
A similar estimate can be obtained for the fully nonlocal operator \eqref{space_time_operator}, i.e. when $A$ is the realization of the fractional Laplacian of order $s\in(0,1)$
\[
\|\partial_{(0,t]}^{\beta} u\|_{\mathcal{C}^{\frac{\gamma}{\beta},\frac{\gamma}{2s}}(Q)}+\|(-\Delta)^s u\|_{\mathcal{C}^{\frac{\gamma}{\beta},\frac{\gamma}{2s}}(Q)}\leq C(\|f\|_{C^\frac{\gamma}{2s}(I;C^{\frac{\gamma}{\beta}}(Q))}+\|u_0\|_{C^{2+\frac{\gamma}{\beta}}(\T)})\ .
\]
We now provide a proof of the H\"older's regularity estimate \eqref{heat;3}  which exploits   formula \eqref{duhamel} and the K-method introduced in Section \ref{susb}. This approach is inspired by the tools used in \cite{LunardiNote} (see also \cite{LunardiBook,Sinestrari}) which however requires to restrict the range of $\beta\in(1/2,1)$. We remind the reader that Schauder type estimates for the classical heat operator go back to \cite{LSU} (see also \cite{LunardiSNS}), while for the space-fractional heat equation we refer the reader to \cite{CF,FRRO,CG1}. 

We begin with some preliminary decay estimates for the fractional heat semigroup $e^{t\Delta}$ in H\"older spaces.
\begin{lemma}\label{EstHolFracHeat}
For every $\theta_1,\theta_2\in\R$, $0\leq\theta_1<\theta_2$,  there exists $C=C(\theta_1,\theta_2)$ such that for all $f\in C^{\theta_1}(\T)$
  \[
  \|e^{t\Delta }f\|_{C^{\theta_2}(\T)}\leq Ct^{-(\theta_2-\theta_1)/2} \|f\|_{C^{\theta_1}(\T)}\ .
  \]
\end{lemma}
\begin{proof}
Computations of \cite[Remark 2.7]{CG1} (in particular the representation formula for heat semigroup  and Young's inequality for convolution) show that for every $k>h$, $k,h\in\N\cup\{0\}$ there exists $C=C(h,k)$
\[
\|e^{t\Delta}f\|_{C^{k+h}(\T)}\leq Ct^{-\frac{k}{2}}\|f\|_{C^h(\T)}\ .
\]
This implies that $e^{t\Delta}f:C^{h}(\T)\to C^{k+h}(\T)$ is bounded for $t > 0$. Recall that, as a consequence of the so-called Reiteration Theorem \cite[Section 1.2.4]{LunardiBook} and \cite[Theorem 1.1.14 and Example 1.1.7]{LunardiNote} (whose proofs can be readily adapted to the torus) we get
\[
(C^h(\T),C^{k+h}(\T))_{\alpha,\infty}=C^{h+\alpha}(\T)\ .
\]
In addition, one also has $e^{t\Delta}f:L^{\infty}(\T)\to L^{\infty}(\T)$. By interpolation (see \cite[Proposition 1.2.6]{LunardiBook}), $e^{t\Delta}$ maps $C^{\theta_1}(\T)$ onto $C^{\theta_2}(\T)$ with the desired estimate. 
\end{proof}
We also recall the following interpolation result.
\begin{lemma}\label{interpholder}
Let $0\leq\theta_1<\theta_2$. For $\sigma\in(0,1)$ such that $(1-\sigma)\theta_1+\sigma\theta_2$ is not an integer, we have
\[
(C^{\theta_1}(\T),C^{\theta_2}(\T))_{\sigma,\infty}=C^{(1-\sigma)\theta_1+\sigma\theta_2}(\T)
\]
\end{lemma}
\begin{proof}
The proof is a consequence of the Reiteration Theorem \cite[Section 1.2.4]{LunardiBook} 
\end{proof}

\begin{thm}\label{Schauder}
Let $\beta\in(1/2,1)$, $\gamma\in(0,1)$ with $\gamma\neq\beta$, $f\in\mathcal{C}^{\frac{\gamma}{\beta},\frac{\gamma}{2}}(Q)$ and $u_0\in C^{2+\frac{\gamma}{\beta}}(\T)$. Then, there exists a constant $C$, depending on $d,T,\beta,\gamma$ (which remains bounded for bounded values of $T$) such that every classical solution to \eqref{i;2} fulfills
\[
\sup_{t\in[0,T]}\|u(\cdot,t)\|_{C^{2+\frac{\gamma}{\beta}}(\T)}\leq C(\sup_{t\in[0,T]}\|f(\cdot,t)\|_{C^{\frac{\gamma}{\beta}}(\T)}+\|u_0\|_{C^{2+\frac{\gamma}{\beta}}(\T)})
\]
\end{thm}
\begin{proof}
The key point is to use Lemma \ref{interpholder}, which gives the identity
\[
C^{2+\frac{\gamma}{\beta}}(\T)=(C^{\frac{\gamma}{\beta}+\delta}(\T),C^{\frac{2}{\beta}+\frac{\gamma}{\beta}+\delta}(\T))_{\beta-\delta\beta/2,\infty}.
\]
We take $\delta<4-2/\beta<2$ (note that $4-2/\beta>0$ since $\beta>1/2$) in order to have that $\beta-\delta\beta/2\in(0,1)$ and $1-\delta/2-1/\beta>-1$. 
We write
\begin{multline*}
u(t)=E_\beta(-t^\beta\Delta)u_0+\int_0^{t}\tau^{\beta-1}E_{\beta,\beta}(-\tau^\beta\Delta)f(t-\tau)d\tau\\
=\int_0^\infty M_\beta(\eta)e^{-(t^\beta\eta-\min\{\xi^\beta\eta, t^\beta\eta\})\Delta}e^{-\min\{\xi^\beta\eta,t^\beta\eta\}\Delta}u_0d\eta\\
+\int_0^{\min\{\xi,t\}}\tau^{\beta-1}E_{\beta,\beta}(-\tau^\beta\Delta)f(t-\tau)d\tau+\int_{\min\{\xi,t\}}^t \tau^{\beta-1} E_{\beta,\beta}(-\tau^\beta\Delta)f(t-\tau)d\tau\\
=c(\xi)+a(\xi)+b(\xi)\ .
\end{multline*}
Then $a(\xi)\in C^{\frac{\gamma}{\beta}+\delta}(\T)$ and $b(\xi),c(\xi)\in C^{\frac{2}{\beta}+\frac{\gamma}{\beta}+\delta}(\T)$. Indeed, by Lemma \ref{EstHolFracHeat}
\begin{multline*}
\|a(\xi)\|_{C^{\frac{\gamma}{\beta}+\delta}(\T)}\leq \beta \int_0^{\min\{\xi,t\}}\tau^{\beta-1}\int_0^\infty M_\beta(\eta)\eta\|e^{-\eta\tau^{\beta}\Delta}f\|_{C^{\frac{\gamma}{\beta}+\delta}(\T)}d\eta d\tau\\
\leq  \beta\int_0^{\min\{\xi,t\}}\tau^{\beta-1}\int_0^\infty M_\beta(\eta)\eta^{1-\delta/2}d\eta\ \tau^{-\delta\beta/2}d\tau\sup_{t\in [0,T]}\|f(\cdot,t)\|_{C^{\frac{\gamma}{\beta}}(\T)}\\
\leq C_1\xi^{\beta-\delta\beta/2}\sup_{t\in [0,T]}\|f(\cdot,t)\|_{C^{\frac{\gamma}{\beta}}(\T)}\ ,
\end{multline*}
where we used that $\int_0^\infty M_\beta(\eta)\eta^{1-\delta/2}d\eta<\infty$ by Proposition \ref{propM} and $1-\delta/2>-1$.
Moreover
\begin{multline*}
\|b(\xi)\|_{C^{\frac{2+\gamma}{\beta}+\delta}(\T)}\leq \beta\int_{\min\{\xi,t\}}^t \tau^{\beta-1}\int_0^\infty M_\beta(\eta)\eta\|e^{-\eta\tau^{\beta}\Delta}f\|_{C^{\frac{2+\gamma}{\beta}+\delta}(\T)}d\eta d\tau\\
\leq  \beta\int_{\min\{\xi,t\}}^t \tau^{\beta-1}\int_0^\infty M_\beta(\eta)\eta^{1-1/\beta-\delta/2}d\eta\ \tau^{-1-\delta\beta/2}d\tau\sup_{t\in[0,T]}\|f(\cdot,t)\|_{C^{\frac{\gamma}{\beta}}(\T)}\\
\leq C\xi^{\beta-1-\delta\beta/2}\sup_{t\in[0,T]}\|f(\cdot,t)\|_{C^{\frac{\gamma}{\beta}}(\T)}\ .
\end{multline*}
Finally
\begin{multline*}
\|c(\xi)\|_{C^{\frac{2+\gamma}{\beta}+\delta}(\T)}\leq \int_0^\infty M_\beta(\eta)\|e^{-\eta\xi^\beta\Delta}u_0\|_{C^{\frac{2}{\beta}+\frac{\gamma}{\beta}+\delta+2-2}(\T)}\\
\leq C\int_0^\infty M_\beta(\eta)\eta^{1-\delta/2-1/\beta}d\eta\|u_0\|_{C^{2+\frac{\gamma}{\beta}}(\T)}\xi^{\beta-1-\delta\beta/2}\leq C\xi^{\beta-1-\delta\beta/2}\|u_0\|_{C^{2+\frac{\gamma}{\beta}}(\T)}.
\end{multline*}
Therefore
\begin{multline*}
\xi^{-(\beta-\delta\beta/2)}K(\xi,u(t),C^{\frac{\gamma}{\beta}+\delta}(\T),C^{\frac{2}{\beta}+\frac{\gamma}{\beta}+\delta}(\T))\leq \xi^{-(\beta-\delta\beta/2)}\{a(\xi)+\xi(b(\xi)+c(\xi))\}\\
\leq C(\sup_{t\in[0,T]}\|f(\cdot,t)\|_{C^{\frac{\gamma}{\beta}}(\T)}+\|u_0\|_{C^{2+\frac{\gamma}{\beta}}(\T)}).
\end{multline*}
\end{proof}

\begin{rem}
One can also study parabolic Schauder estimates for the operator $\partial_{(0,t]}^{\beta} +(-\Delta)^s$ and $\partial_t-\Delta +(-\Delta)^s$. In the second case, one has to use the fact that the semigroup generated by the sum of $-\Delta +(-\Delta)^s$ is the composition of the two semigroups  and get the right decay bounds (see   \cite{CG1} and references therein). 

\end{rem}
\section{On time-fractional Hamilton-Jacobi equation}\label{sec:existence}
In this section, we prove existence and regularity results for the  time-fractional Hamilton-Jacobi equation \eqref{i;1}. In the first part, via nonlinear adjoint method, we get an a priori bound on the gradient of the solution. In the second part, exploiting the previous bound, we prove existence and uniqueness of classical solution to the problem.\\
From now on, we suppose that $H=H(x,p)$ is $C^2(\T\times\R^d)$, $H(x,p)\geq H(x,0)=0$ (if not, one may compensate by adding a positive constant to $V$) and there exist constants $\gamma > 1$ and $c_H,C_H,\tilde{C}_H>0$   such that
\begin{align}
\tag{H1}\label{H1} & D_pH(x,p)\cdot p-H(x,p)\geq C_H|p|^{\gamma}-c_H\ , \\
\tag{H2}\label{H2} & |D_pH(x,p)|\leq C_H|p|^{\gamma-1}+\tilde{C}_H \, \\
\tag{H3}\label{H3} & |D_{xx}^2H(x,p)|\leq C_H|p|^{\gamma}+\tilde{C}_H \ , \\
\tag{H4}\label{H4} & |D_{px}^2H(x,p)|\leq C_H|p|^{\gamma-1}+\tilde{C}_H \ , \\
\tag{H5}\label{H5} & D_{pp}^2H(x,p)\xi\cdot \xi\geq C_H|p|^{\gamma-2}|\xi|^2-\tilde{C}_H
\end{align}
for every $x\in \T$, $p\in\R^d$ and $\xi\in\R^d$. Typical examples of Hamiltonians we have in mind are related to the theory of Mean Field Games, where $H$ is convex and superlinear in the second entry. For instance, $H_1(x,p)=h_1(x)|p|^2$ or $H_2(x,p)=h_2(x)\{(1+|p|^2)^{\frac{\gamma}{2}}-1\}$ with $h_i\in C^2(\T)$, $i=1,2$, fulfill the above assumptions.


\subsection{Well-posedness and regularity results for the time-fractional Fokker-Planck equation}

We will use the following integration by parts formula, whose proof can be found in e.g. \cite[Lemma 2.8]{LinLiu} 
\begin{lemma}\label{intparts}
Let $u,v\in C^1([0,\tau])$. Then
\[
\int_0^\tau(\partial_{(0,t]}^\beta u(t))(v(t)-v(\tau))\,dt=\int_0^\tau (\partial_{[t,\tau)}^\beta v(t))(u(t)-u(0))\,dt
\]
In particular, this is equivalent to
\[
\int_0^\tau v(t)\partial_{(0,t]}^\beta u(t)\,dt+u(0) (I^{1-\beta}_{[0,\tau)}v) (0)=\int_0^\tau u(t)\partial_{[t,\tau)}^\beta v(t)\,dt+v(\tau) (I^{1-\beta}_{(0,\tau]}u) (\tau)\ ,
\]
where
\begin{align*}
(I^{1-\beta}_{(0,t]} u) ( t)=\frac{1}{\Gamma(1-\beta)}\int_0^{t}\frac{u(s)}{(t-s)^\beta}ds,\\
(I^{1-\beta}_{(t,\tau]} v) ( t)=\frac{1}{\Gamma(1-\beta)}\int_t^{\tau}\frac{v(s)}{(s-t)^\beta}ds,
\end{align*}
are the forward and backward Riemann-Liouville integrals. The integration by parts formula remains true if $u\in L^1_{loc}(0,\tau)$ for $v\in C_c^\infty((-\infty,\tau))$.
\end{lemma}

\begin{proof}
The first identity is proved in \cite[Lemma 2.8 and Remark 2.6]{LinLiu}. The second one is a consequence of the equalities
\[
\partial_{(0,t]}^\beta u(t)=\frac{d}{dt}\left[I^{1-\beta}_{(0,t]} u(t)\right]-\frac{u(0)}{t^\beta\Gamma(1-\beta)}\ ,
\]
\[
\partial_{[t,\tau)}^\beta v(t)=-\frac{d}{dt}\left[I^{1-\beta}_{[t,\tau)}v(t)\right]-\frac{v(\tau)}{(\tau-t)^\beta\Gamma(1-\beta)}\ .
\]
\end{proof}
We consider now the (backward) time-fractional Fokker-Planck equation
\begin{equation}\label{adjoint}
\begin{cases}
\partial_{[t,\tau)}^\beta\rho-\sigma\Delta \rho-\mathrm{div}(b(x,t)\rho)=0&\text{ in }Q_\tau:=\T\times(0,\tau)\ ,\\
\rho(x,\tau)=\rho_\tau(x)&\text{ in }\T,
\end{cases}
\end{equation}
where $\sigma$ is a positive constant and
\[
\partial_{[t,\tau)}^\beta \rho(x,t)=-\frac{1}{\Gamma(1-\beta)}\int_t^\tau\frac{ \partial_s\rho(x,s)}{(s-t)^\beta}\,ds ,
\]
is the backward Caputo derivative, $\tau<T$. We first premise the following 
\begin{defn}
Let $\beta\in(0,1)$ and $b\in L^\infty(Q_\tau;\R^d)$. We say that $\rho\in L^2(0,\tau;H^1(\T))$ with $\partial_{[t,\tau)}^\beta \rho\in L^2(0,\tau;H^{-1}(\T))$ is a weak solution to \eqref{adjoint} in the sense that
\[
\int_0^\tau\int_\T\rho \partial_{(0,t]}^{\beta} \varphi+\sigma \rho\Delta \varphi-b\cdot D\varphi\rho\,dxdt=\int_0^\tau\int_\T\rho_\tau\partial_{(0,t]}^{\beta} \varphi\,dxdt\ .
\]
for every $\varphi\in C^\infty(\T\times(0,\tau])$.
\end{defn}
We recall that it is possible to define a weak notion of Caputo-time derivative as in \cite{LinLiu}. Note also that compactly supported test functions are enough by density arguments. In particular, it holds
\[
\int_0^\tau \langle \partial_{[t,\tau)}^\beta\rho,\varphi\rangle_{H^{-1}(\T),H^1(\T)} +\int_0^\tau\int_\T \sigma D\rho\cdot D\varphi-b\cdot D\varphi\rho\,dxdt=0,
\]
for every $\varphi\in L^2(0,\tau;H^1(\T))$, $\langle \cdot,\cdot \rangle_{H^{-1}(\T),H^1(\T)}$ being the duality between $H^1(\T)$ and $H^{-1}(\T)$. We emphasize that the requirement $\partial_{[t,\tau)}^\beta \rho\in L^2(0,\tau;H^{-1}(\T))$ guarantees that $\rho$ is continuous at the terminal time whenever $\beta>1/2$, so that the left limit $\rho(x,\tau^-)=\rho_\tau(x)$ is well-defined, see e.g. \cite[Lemma 3.1-(iii)]{LinLiu}. Note that this restriction on $\beta$ will be needed even in the context of the Hamilton-Jacobi equation later on.

We prove the following existence result (a similar result appeared in \cite[Theorem 2.8]{TangDCDS}.
\begin{prop}\label{well}
Let $b\in L^\infty(Q_T;\R^d)$, $\rho_\tau\in L^2(\T)$ with $\rho_\tau\geq0$ and
\[
W=\{\rho\in L^2(0,\tau;H^1(\T)),\,\partial_{[t,\tau)}^\beta \rho\in L^2(0,\tau;H^{-1}(\T)).
\}
\]
Then, there exists a unique weak solution to \eqref{adjoint} which belongs to $W
$. In particular, we also have $\rho\geq0$ a.e. on $Q_\tau$ and $\|\rho(t)\|_{L^1(\T)}=\|\rho_\tau\|_{L^1(\T)}$ for a.e. $t\in[0,\tau)$. 
\end{prop}
\begin{proof}
We set $\sigma=1$ for simplicity.  Using the convexity of the $L^2$-norm one has
\begin{equation}\label{claim}
\frac12 \partial_{[t,\tau)}^\beta\rho^2(x,t)\leq \partial_{[t,\tau)}^\beta\rho(x,t) \rho(x,t)
\end{equation}
for a.e. $(x,t)\in Q_\tau$ (see e.g. \cite[Lemma 2.4]{LinLiu}).
We can now use $\varphi=\rho$ as a test function in the weak formulation, \eqref{claim}, Lemma \ref{intparts} and Young's inequality to conclude
\[
\frac12 \partial_{[t,\tau)}^\beta\|\rho(t)\|_{L^2(\T)}^2+ \int_\T |D\rho|^2\,dx
\leq \frac{1}{2}\int_\T |D\rho|^2\,dx+C\|\rho_\tau\|_{L^2(\T)}\ ,
\]
where we also used that $\|\rho(s)\|_{L^2(\T)}\leq C\|\rho_\tau\|_{L^2(\T)}$ for a.e. $s\in(0,\tau)$ (these are classical estimates that can be obtained, for instance, using the Duhamel-like formula \eqref{duhamel} for suitably regular drift vector fields, see also \cite[Lemma 2.1]{TangDCDS}),
giving in particular
\begin{equation}\label{en}
\frac12 \partial_{[t,\tau)}^\beta\|\rho(t)\|_{L^2(\T)}^2+\frac{1}{2}\|D\rho\|_{L^2(\T)}^2\leq C\|\rho_\tau\|_{L^2(\T)}
\end{equation}
for some positive constant $C>0$. 
Using \cite[Lemma 2.3]{LinLiu} in \eqref{en}, one obtains the non-local control
\[
\sup_{t\in(0,\tau)}\frac{1}{\Gamma(\beta)}\int_0^t(t-s)^{1-\beta}(\|\rho(s)\|_{H^1(\T)}^2)\,ds\leq C\ .
\]
This yields an a priori estimate of $\rho\in L^2(H^1)$ and $\sup_{t\in(0,\tau)}I^{1-\beta}_{(0,t]}(\|\rho(t)\|_{H^1(\T)}^2)$ in terms of the data. Using these bounds, the estimate of $\partial_{[t,\tau)}^\beta\rho$ in $L^2(H^{-1})$ follows immediately by duality. \\
To prove the existence, we can argue using e.g. the Leray-Schauder fixed point theorem on the space $W$, which are Banach spaces due to \cite{KKL}. To this aim, consider the map $G:W\times[0,1]\to W$ defined by $m\longmapsto \rho=G[m;\zeta]$ as the solution to the parametrized problem
\[
\partial_{[t,\tau)}^\beta\rho-\Delta \rho=\zeta\mathrm{div}(b(x,t)m)\ \text{ in }Q_\tau\ ,\rho(x,\tau)=\zeta\rho_\tau ,\text{ in }\T\ .
\]
We note that $G[m;0]=0$ by standard results for the time-fractional heat equation. To show that the map is well-defined, one can argue as in the previous procedure by testing the equation against $\rho$ itself (see also Remark \ref{incrreg}). Furthermore, if $\rho\in W$ and $\zeta\in[0,1]$ satisfy $\rho=G[\rho;\zeta]$, then the a priori estimates carry through uniformly in $\zeta$, giving $\|\rho\|_W\leq C$.\\
The compactness of the map $G$ follows by using the compactness criteria in \cite{LinLiu}. More precisely, let $m_n\in W$ be a bounded sequence such that $\rho_n=G[m_n;\zeta]$. Using the strategy outlined above, we obtain that $\rho_n$ is bounded in $W$ and the nonlocal control
\[
\sup_{t\in(0,\tau)}\frac{1}{\Gamma(\beta)}\int_0^t(t-s)^{1-\beta}(\|\rho_n(s)\|_{H^1(\T)}^2)\,ds\leq C\ .
\]
By \cite[Theorem 4.2]{LinLiu} (with $M=H^1(\T)$, $B=L^2(\T)$ and $Y=H^{-1}(\T)$) we conclude the strong convergence of $\rho_n$ to $\rho$ in $L^2(Q_\tau)$. Thus, without relabeling the index, we call $\rho_n$ the (sub)sequence which converges strongly to $\rho$ in $L^2(Q_\tau)$ and such that $D\rho_n$ converges weakly to $D\rho$ in $L^2(Q_\tau)$. We take $\varphi=\rho_n-\rho$ as a test function in the weak formulation of the equation satisfied by $\rho_n$ to get
\[
\int_0^\tau\int_\T\rho_n \partial_{(0,t]}^{\beta} (\rho_n-\rho)+D\rho\cdot D \varphi-b\cdot D(\rho_n-\rho)m_n\,dxdt=\int_0^\tau\int_\T\rho_\tau\partial_{(0,t]}^{\beta}(\rho_n-\rho)\,dxdt\ .
\]
and then, using \eqref{claim}, we get
\begin{multline*}
\int_0^\tau\int_\T \partial_{(0,t]}^{\beta}\frac12 \|(\rho_n-\rho)(t)\|_{L^2(\T)}^2\,dxdt+\int_0^\tau\int_\T|D(\rho_n-\rho)|^2\,dxdt\leq-\int_0^\tau\int_\T \rho \partial_{(0,t]}^{\beta}(\rho_n-\rho)\,dxdt\\
- \int_0^\tau\int_\T D\rho\cdot D(\rho_n-\rho)\,dxdt+\int_0^\tau\int_\T\rho_\tau(x)\partial_{(0,t]}^{\beta}(\rho_n-\rho)\,dxdt+\int_0^\tau\int_\T b\cdot D(\rho_n-\rho)m_n\,dxdt\\
=(I)+(II)+(III)+(IV)
\end{multline*}
Note that  (II)-(IV) converge to 0 using the strong convergence of $\rho_n$ to $\rho$ in $L^2(Q_\tau)$ and the weak convergence of $D\rho_n$ to $D\rho$ in $L^2(Q_\tau)$. Moreover, (I)-(III) converge to 0 integrating in time $\partial_{(0,t]}^{\beta}(\rho_n-\rho)$, exploiting the strong convergence of $\rho_n$ to $\rho$ in $L^2(Q_\tau)$ and the fact that $\rho_\tau\in L^2(\T)$. Using again \cite[Lemma 2.3]{LinLiu}, we observe that
\[
\frac{1}{\Gamma(\beta)}\int_0^t(t-s)^{1-\beta}(\|D(\rho_n-\rho)(s)\|^2_{L^2(\T)})\,ds\to 0 
\]
as $n\to\infty$, giving the strong convergence of $D\rho_n$ to $D\rho$ in $L^2(Q_\tau)$. The convergence of the fractional time derivative follows by duality.


The uniqueness of solutions can be obtained, for instance, using the existence for the dual equation, which holds true due to the boundedness of the drift vector field (see e.g. \cite{Zacherabstract}) .

\end{proof}
\begin{rem}\label{incrreg} In particular, it can be proved by similar arguments to those used in \cite{CT} that $\rho$ enjoys better regularity properties, e.g. in the space of functions \[
\{\rho\in L^p(0,\tau;W^{1,p}(\T))\ ,\partial_{[t,\tau)}^\beta \rho\in L^p(0,\tau;W^{-1,p}(\T))\}
\]
for $p>1$. We sketch the proof here for reader's convenience. 
By the definition of weak solution, we have
\begin{multline}\label{weakid}
\left|\iint_{Q_\tau} \rho (\partial_{(0,t]}^{\beta} \varphi-\Delta\varphi)\,dxdt\right|\leq \left|\iint_{Q_\tau}\rho_\tau\partial_{(0,t]}^{\beta} \varphi\,dxdt\right|\\
+\iint_{Q_\tau} \rho |b| |D\varphi|\,dxdt\leq \left|\iint_{Q_\tau}\rho_\tau\partial_{(0,t]}^{\beta} \varphi\,dxdt\right|+\|b\|_{L^\infty(Q_\tau)}\|D\varphi\|_{L^{p'}(Q_\tau)}\|\rho\|_{L^p(Q_\tau)}.
\end{multline}
For $i=1,...,d$ consider the strong solution to the following inhomogeneous time-fractional heat equation
\begin{equation}\label{aux}
\begin{cases}
\partial_{(0,t]}^{\beta}\psi-\Delta\psi =|\partial_{x_i}\rho|^{p-2}\partial_{x_i}\rho &\text{ in }Q_\tau:=\T\times(0,\tau)\ ,\\
\psi(x,\tau)=0&\text{ in }\T\ .
\end{cases}
\end{equation}
By $L^p$ maximal regularity for abstract evolution equations (cf Theorem \ref{thm;maxreg}) we get
\[
\|\psi\|_{\Vb_{p'}^2(Q_\tau)}\leq \tilde C\||\partial_{x_i}\rho|^{p-2}\partial_{x_i}\rho\|_{L^{p'}(Q_\tau)}=\tilde C\|\partial_{x_i}\rho\|_{L^p(Q_\tau)}^{p-1}\ .
\]
A straightforward application of the H\"older's inequality yields
\[
\left| \iint_{Q_\tau}\partial_{x_i}\rho_\tau\partial_{(0,t]}^{\beta}\psi\,dxdt\right|\leq \tau^\frac1p\|\partial_{x_i}\rho_\tau\|_{L^p(\T)}\|\partial_{(0,t]}^{\beta}\psi\|_{L^{p'}(Q_\tau)}\\
\leq \tau^\frac1p\|\partial_{x_i}\rho_\tau\|_{L^p(\T)}\|\psi\|_{\Vb_{p'}^2(Q_\tau)}.
\]
We take $\varphi=\partial_{x_i}\psi$ in \eqref{weakid} and, after integrating by parts, we have
\begin{multline*}
\left|\iint_{Q_\tau}\partial_{x_i}\rho(\partial_{(0,t]}^{\beta}\psi-\Delta\psi)\,dxdt\right|\leq  \left|\iint_{Q_\tau}\partial_{x_i}\rho_\tau\partial_{(0,t]}^{\beta} \psi\,dxdt\right|+\|b\|_{L^\infty(Q_\tau)}\|D(\partial_{x_i}\psi)\|_{L^{p'}(Q_\tau)}\|\rho\|_{L^p(Q_\tau)}\\
\leq \tau^\frac1p\|D\rho_\tau\|_{L^p(\T)}\|\psi\|_{\Vb_{p'}^2(Q_\tau)}+\|b\|_{L^\infty(Q_\tau)}\|\psi\|_{\Vb_{p'}^2(Q_\tau)}\|\rho\|_{L^p(Q_\tau)}
\leq C\|\psi\|_{\Vb_{p'}^2(Q_\tau)}.
\end{multline*}
Using the equation satisfied by $\psi$ we find
\begin{multline*}
\iint_{Q_\tau}|\partial_{x_i}\rho|^p\,dxdt=\iint_{Q_\tau} \partial_{x_i}\rho |\partial_{x_i}\rho|^{p-2}\partial_{x_i}\rho\,dxdt\\
=\left|\iint_{Q_\tau}\partial_{x_i}\rho(\partial_{(0,t]}^{\beta}\psi-\Delta\psi )\,dxdt\right|\leq C\|\psi\|_{\Vb_{p'}^2(Q_\tau)}\leq C\tilde C\|\partial_{x_i}\rho\|_{L^p(Q_\tau)}^{p-1}
\end{multline*}
giving thus the estimate on $D\rho\in L^p(Q_\tau)$. As a byproduct, we conclude by Poincar\`e inequality
\[
\|\rho\|_{L^p(Q_\tau)}\leq \|\rho-1\|_{L^p(Q_\tau}+\tau\leq C_1\|D\rho\|_{L^p(Q_\tau)}+\tau
\]
which can be bounded in terms of the data by the above estimates. The bound on the Caputo derivative follows by duality and using the equation.\\
We remark in passing that such strategy can be implemented even when low regularity information on the drift against the density $\rho$ are available,  namely when $b\in L^k(\rho)$ for some $k>1$, see e.g. \cite{CT,CG2,MPR}.\\
Finally, when the drift $b$ is smooth and $\rho\in L^p(W^{1,p})$ for every $p>1$, one achieves classical regularity by using maximal regularity results and implementing bootstrapping procedures.
\end{rem}

\subsection{Gradient bound}
We consider the time-fractional Hamilton-Jacobi equation
\begin{equation}\label{xfracHamilton-Jacobisemi}
\begin{cases}
\ds \partial_{(0,t]}^\beta u(x,t) -\sigma\Delta u+ H(x, Du(x,t)) = V(x,t) & \text{in $Q_T$,} \\
u(x, 0) = u_0(x) & \text{in $\T$,}
\end{cases}
\end{equation}
where $\sigma$ is a positive constant. In this section, we prove 
gradient bound for classical solution to \eqref{xfracHamilton-Jacobisemi}.
The method implemented here is based on the so-called nonlinear adjoint method (see \cite{CG2,gomesbook} and \cite[Proposition 3.6]{CG1}), that is on testing the Hamilton-Jacobi equation against the (classical) solution to \eqref{adjoint} with optimal drift $b(x,t):=-D_pH(x,Du(x,t)$, i.e.
\begin{equation}\label{adjointmfg}
\begin{cases}
\partial_{[t,\tau)}^\beta\rho-\sigma\Delta \rho-\mathrm{div}(D_pH(x,Du)\rho)=0&\text{ in }Q_\tau:=\T\times(0,\tau)\ ,\\
\rho(x,\tau)=\rho_\tau(x)&\text{ in }\T,
\end{cases}
\end{equation}
where
$\tau<T$. Here, we assume $\rho_\tau\in C^\infty(\T)$ with $\rho_\tau\geq0$ and $\int_{\T}\rho_\tau(s)ds=1$. Under these assumptions, the terminal data $\rho_\tau$ should be thought as an item of a sequence approximating a singular terminal data (see e.g. \cite{Evans}). Time-fractional Fokker-Planck equations has been recently analyzed in \cite{KZ}. We first recall  an a priori bound on the sup-norm for linear time-fractional PDEs.
\begin{lemma}\label{mp}
Let $u$ be a classical solution to
\begin{equation}\label{gb;1}
\begin{cases}
 \partial_{(0,t]}^\beta u-\sigma\Delta u+b(x,t)\cdot Du= F(x,t) & \text{in $Q_T$,} \\
u(x, 0) = u_0(x) & \text{in $\T$,}
\end{cases}
\end{equation}
with $b\in C(Q_T;\R^d)$, $F\in C(Q_T)$ and $u_0\in C(\T)$. Then, it holds
\[
\|u\|_{\infty;Q_T}\leq \|u_0\|_{\infty,\T}+\frac{T^\beta}{\Gamma(1+\beta)}\|F\|_{\infty;Q_T}
\]
\end{lemma}
\begin{proof}
The idea is to use the same procedure as in \cite[Theorem 4]{Luchko}. Consider the function
\[
w(x,t)=u(x,t)-\frac{M}{\Gamma(1+\beta)}t^{\beta}
\]
with $M:=\|F\|_{\infty;Q_T}$ and note that it is a classical solution to
\[
\partial_{(0,t]}^\beta w-\sigma\Delta w+b\cdot Dw=F_1,
\]
where $F_1(x,t)=F(x,t)-M$, since it holds
\[
\partial_{(0,t]}^\beta t^\gamma=\frac{\Gamma(1+\gamma)}{\Gamma(1-\beta+\gamma)}t^{\gamma-\beta}\ ,\gamma\in\R\ ,\gamma>0\ .
\]
(see \cite[eq. (A.15)]{TaylorNote}). In particular, note that $F_1\leq 0$ and, by the maximum principle (see e.g. \cite[Theorem 2]{Luchko}, where similar arguments can be used for the case at hand of the drift-diffusion operator $\partial_{(0,t]}^\beta-\sigma\Delta+b\cdot D$), we have
\[
w(x,t)\leq \|u_0\|_{\infty,\T}\ ,
\]
allowing to conclude
\[
u(x,t)\leq  \|u_0\|_{\infty,\T}+\frac{T^\beta}{\Gamma(1+\beta)}\|F\|_{\infty;Q_T}.
\]
The lower bound follows by applying a minimum principle to the function $w(x,t)=u(x,t)+\frac{M}{\Gamma(1+\beta)}t^{\beta}$.
\end{proof}

\begin{lemma}\label{reprsemic}
Let $u$ be a classical solution to \eqref{xfracHamilton-Jacobisemi}. Then, there exists a (nonnegative) classical solution $\rho$ to \eqref{adjoint}. Moreover, we have the estimate
\[
\int_{0}^\tau\int_{\T}|Du|^{\gamma}\rho \, dxdt \le C,
\]
where $C$ depends on $K$ and not on $\rho_\tau$ nor $\tau$.
\end{lemma}

\begin{proof}
By multiplying the time-fractional Hamilton-Jacobi  equation
\[
\partial_{(0,t]}^\beta u-\sigma \Delta u+H(x,Du)=V
\]
by $-\rho$ and the adjoint equation by $u$, using Lemma \ref{intparts} and summing both expressions one easily obtains the following formula
\begin{equation}\label{reprformula}
\int_{\T}(I^{1-\beta}_{(0,\tau]}u)(x,\tau)\rho(x,\tau)dx=\int_{\T}u(x,0)(I^{1-\beta}_{[0,\tau)}\rho)(x,0)dx +\iint_{Q_\tau}V\rho \, dxdt+
\end{equation}
\begin{equation*}
+\iint_{Q_\tau}(D_pH(x,Du)\cdot Du-H(x,Du))\rho \, dxdt.
\end{equation*}
 Then, by \eqref{H1}, we get
\begin{multline}\label{ineq1}
\int_{\T}(I^{1-\beta}_{(0,\tau]}u)(x,\tau)\rho(x,\tau)dx\geq \int_{0}^{\tau}\int_{\T}V\rho\,dxdt +C_H\int_{0}^\tau\int_{\T}|Du|^{\gamma}\rho\,dxdt-\\-c_H\int_{0}^\tau\int_{\T}\rho\,dxdt
+\int_{\T}u(x,0)(I^{1-\beta}_{[0,\tau)}\rho)(x,0)dx\ .
\end{multline}
Since $u$ is a classical solution to \eqref{xfracHamilton-Jacobisemi}, a standard linearization argument and the application of the Comparison Principle for time-fractional viscous PDE (see Lemma \ref{mp}) yield
\begin{equation}\label{comp}
\norm{u}_{\infty;Q_T}\leq \norm{u_0}_{\infty;\T} + \frac{T^\beta}{\Gamma(1+\beta)}\big(\norm{V}_{\infty;Q_T} + \norm{H(\cdot,0)}_{\infty;\T}\big).
\end{equation}
Finally, using the facts that
\begin{align*}
\|(I^{1-\beta}_{(0,\tau]}u)(x,\tau)\|_{L^\infty(\T)}\leq C\|u(x,\tau)\|_{L^\infty(\T)},\\
\|(I^{1-\beta}_{[0,\tau)}\rho)(x,0))\|_{L^1(\T)}\leq C_1 \|\rho(x,0)\|_{L^1(\T)},
\end{align*}
for all $t$ (see e.g. \cite{KKL}), and plugging \eqref{comp} in \eqref{ineq1} we conclude the desired estimate.
\end{proof}
\begin{thm}\label{semicest}
Assume that $V \in C([0,T];C^{2+\frac{\gamma}{\beta}}(\T))$, $u_0\in C^2(\T)$ and let   $K > 0$ be such that
\[
\norm{V}_{C^{2}_x(Q_T)} + \|u_0\|_{C^2(\T)} \leq K.
\]
Then a  classical solution $u$ to \eqref{xfracHamilton-Jacobisemi} satisfies
\[
\|I^{1-\beta}(D^2u(x,t)\xi\cdot\xi)\|_{L^\infty(\T)}\leq C_1
\]
for $\xi\in\R^d$, $|\xi|=1$, where $C_1$ depends on $K$. In particular, it follows that $u$ enjoys the estimate
\begin{equation}\label{gradientbound}
\|Du\|_{L^p(Q_\tau)}\leq C_2
\end{equation}
for every  $p\ge 1$ and some $C_2$ depending in particular on $K$, $\beta$ and not on $\sigma$. Moreover, $C_2$ remains bounded for bounded values of $\tau$.
\end{thm}
\begin{proof}
Let $u$ be a classical solution and $\xi\in\R^d$, $|\xi|=1$. Arguing as in the next Theorem \ref{existenceHamilton-Jacobi} (see also Remark \ref{stclassical}), one can assume that $u$ has enough regularity to perform all the computations below. Set $v=\partial_\xi u$ and $w=\partial_{\xi\xi}^2u$. Then $w$ solves
\begin{multline}\label{eqw}
\partial_{(0,t]}^\beta w-\sigma\Delta w +Dv\cdot D^2_{pp}H(x,Du)\cdot Dv+D_pH(x,Du)\cdot Dw+2D^2_{p\xi}H(x,Du)\cdot Dv\\
+D^2_{p\xi}H(x,Du)\cdot Dv+D^2_{\xi\xi}H(x,Du)=V_{\xi\xi}
\end{multline}
with initial data $w(x,0)=\partial_{\xi\xi}^2u_0(x)$, which can be equivalently rewritten as an abstract Volterra equation of the form
\[
w(t)=\sigma\int_0^tg_{\beta}(t-\tau)\Delta w(t)dt+g_\beta\star F(t)
\]
where
\begin{align*}
F(t)&=V_{\xi\xi}-(Dv\cdot D^2_{pp}H(x,Du)\cdot Dv+D_pH(x,Du)\cdot Dw\\
&+2D^2_{p\xi}H(x,Du)\cdot Dv+D^2_{p\xi}H(x,Du)\cdot Dv+D^2_{\xi\xi}H(x,Du)).
\end{align*}

We test \eqref{eqw} against the adjoint variable $\rho$ solving \eqref{adjointmfg}. Using Lemma \ref{intparts} and integrating by parts in space
we  get
\begin{equation*}
\int_{\T}(I^{1-\beta}_{(0,\tau]}w)(x,\tau)\rho(x,\tau)dx+\iint_{Q_\tau}Dv\cdot D^2_{pp}H(x,Du)Dv \rho\, dxdt=\int_{\T}w(x,0)(I^{1-\beta}_{[0,\tau)}\rho)(x,0)\,dx-
\end{equation*}
\begin{equation*}
-2\iint_{Q_\tau}D_{p\xi}^2H(x,Du)\cdot Dv\ \rho\, dxdt-\iint_{Q_\tau}D_{\xi\xi}^2H(x,Du)\ \rho\, dxdt+\iint_{Q_\tau}V_{\xi\xi}\ \rho\, dxdt\ .
\end{equation*}
On one hand, by \eqref{H5} we have
\begin{equation*}
\iint_{Q_\tau}Dv\cdot D^2_{pp}H(x,Du)Dv\ \rho\, dxdt\geq C_1\iint_{Q_\tau}|Du|^{\gamma-2}|Dv|^2\ \rho\,dxdt-\tilde{C}_1\iint_{Q_\tau}\rho\,dxdt
\end{equation*}
and hence, using also \eqref{H3}-\eqref{H4}, we conclude
\begin{multline*}
\int_{\T}(I^{1-\beta}_{(0,\tau]}w)(x,\tau)\rho(x,\tau)dx+C_1\iint_{Q_\tau}|Du|^{\gamma-2}|Dv|^2\ \rho \, dxdt-\tilde{C}_1\iint_{Q_\tau}\rho\,dxdt  \\ \leq \int_{\T}w(x,0)(I^{1-\beta}_{(0,\tau]}\rho)(x,0)\,dx+C_2\iint_{Q_\tau}|Du|^{\gamma-1}|Dv|\ \rho dxdt+C_3\iint_{Q_\tau}|Du|^{\gamma}\ \rho\, dxdt \\ +(\tilde{C}_2+\tilde{C}_3)\iint_{Q_\tau}\rho\,dxdt+\iint_{Q_\tau}V_{\xi\xi}\rho\, dxdt.
\end{multline*}
Now, we apply Young's inequality to the second term on the right-hand side of the above inequality to get
\begin{equation*}
\iint_{Q_\tau}|Du|^{\gamma-1}|Dv|\ \rho\,dxdt\leq \frac{\epsilon^2}{2}\iint_{Q_\tau}|Du|^{\gamma-2}|Dv|^2\ \rho dxdt+\frac{1}{\epsilon^2}\iint_{Q_\tau}|Du|^{\gamma}\ \rho dxdt.
\end{equation*}
Taking $\epsilon$ so that $C_1=\frac{\epsilon^2}{2}$ we finally obtain the estimate
\begin{multline*}
\int_{\T}(I^{1-\beta}_{(0,\tau]}w)(x,\tau)\rho(x,\tau)dx\leq  \int_{\T}w(x,0)(I^{1-\beta}_{(0,\tau]}\rho)(x,0)\,dx+\left(\frac{1}{2C_1}+C_3\right)\iint_{Q_\tau}|Du|^{\gamma}\ \rho\, dxdt+\\
+\iint_{Q_\tau}V_{\xi\xi}\ \rho\,dxdt+\tilde{C}_4\ .
\end{multline*}
During the above computations $C_i=C_i(C_H)$. After passing to the supremum over $\rho_\tau$ one gets the estimate on
\[
\|I^{1-\beta}_{(0,\tau]}w(x,\tau)\|_{L^\infty(\T)}\leq C_1
\]
which in turn yields a control on the quantity
\begin{equation}\label{semicnonlocal}
I^{1-\beta}_{(0,\tau]}(\|Du\|_{L^\infty(\T)})\leq C_2.
\end{equation}
Indeed, due to the fact that $u$ is $\Z^d$-periodic one has
\[
\|Du(\cdot,\omega)\|_{L^\infty(\T)}\leq C\sup_{x\in\T,|\xi|=1}\partial_{\xi\xi}^2u(\cdot,\omega)
\]
for $\omega\in[0,\tau]$ and then \eqref{semicnonlocal} follows by the definition of $I^{1-\beta}$. In particular, estimate \eqref{semicnonlocal} assures the control on $\|Du\|_{L^p(Q_\tau)}$ for every $p$ since
\begin{multline*}
\iint_{Q_\tau}|Du(x,\omega)|^pdx d\omega\leq \tau^\beta\iint_{Q_\tau}(\tau-\omega)^{-\beta}|Du(x,\omega)|^pdx d\omega\\
=\tau^\beta\Gamma(1-\beta)I^{1-\beta}_{(0,\tau]}(\|Du\|_{L^p(\T)}^p)
\leq C_1\tau^\beta I^{1-\beta}_{(0,t]}(\|Du\|_{L^\infty(\T)}^p)
\leq C_2\tau^\beta.
\end{multline*}

\begin{rem}
We point out that a similar gradient bound can be achieved again by duality simply by considering the equation satisfied by $v=\partial_\xi u$ and assuming the right-hand side $V\in C([0,T];C^{1+\frac{\gamma}{\beta}}(\T))$. Indeed, $v$ solves
\[
\partial_{(0,t]}^\beta v-\sigma\Delta v+D_pH(x,Du)\cdot Dv+D_\xi H(x,Du)=V_\xi
\]
By testing against $\rho$ solving \eqref{adjoint} and integrating by parts one obtains
\[
\int_{\T}(I^{1-\beta}_{(0,\tau]}v)(x,\tau)\rho(x,\tau)+\iint_{Q_\tau}D_\xi H(x,Du)\rho\,dxdt=\int_{\T}v(x,0)(I^{1-\beta}_{(0,\tau]}\rho)(x,0)\,dx+\iint_{Q_\tau}V_\xi\rho\,dxdt
\]
and then one concludes using the bound in Lemma \ref{reprsemic} as above to get the nonlocal-in-time control
\[
I^{1-\beta}_{(0,\tau]}(\|Du\|_{L^\infty(\T)})\leq C
\]
However, in order to tackle problems for first-order Mean Field Games with nonlocal coupling and fractional time-derivative, where typically $C^2$ regularity on the right-hand side $V$ is assumed \cite{CLLP}, we prefer to keep the estimate for $w=\partial_{\xi\xi}u$ since this would yield a further information for the drift of the Fokker-Planck equation. 
\end{rem}


\end{proof}
\subsection{Existence and uniqueness for the time-fractional Hamilton-Jacobi equation}
This section is devoted to the proof of Theorem \ref{existenceHamilton-Jacobi} concerning
existence and uniqueness of classical solutions to \eqref{i;1} with regular right-hand side $V$.
The crucial step in the proof  is the gradient bound of the previous section which allows to extend the   solution from a local to a global time interval.  

\begin{proof}[Proof of Theorem \ref{existenceHamilton-Jacobi}]
Step 1: Local existence on $Q_\tau := \T \times (0,\tau)$ . Let $\tau\leq1$ and
\[
\mathcal{S}_a:=\left\{u\in \Vb_p^{2}(Q_\tau):u(0)=u_0\ ,\|u\|_{\Vb_p^{2}(Q_\tau)}\leq a\ ,p>d+\frac{2}{\beta}\right\}
\]
be the space on which we apply the contraction mapping principle. The parameter $a$ will be chosen large enough. Fix $z\in \Vb_p^{2}(Q_\tau)$, $p>d+\frac{2}{\beta}$ and let $w = Jz$ be the solution of the problem
\begin{equation}\label{frozen}
\begin{cases}
\partial_{(0,t]}^{\beta} w-\Delta w=V(x,t)-H(x,Dz)&\text{ in }Q_\tau\ ,\\
w(x,0)=u_0(x)&\text{ in }\T.
\end{cases}
\end{equation}
We remark that in view of Proposition \ref{emb} and the properties of $H$ the right-hand side of \eqref{frozen} belong to $L^p(Q_\tau)$, $p>1$ and then Theorem \ref{thm;maxreg} implies that \eqref{frozen} admits a unique solution $w\in \Vb_p^{2}(Q_\tau)$ satisfying the following estimate
\[
\|w\|_{\Vb_p^{2}(Q_\tau)}\leq C(\|V\|_{L^p(Q_\tau)}+\|H(x,Dz)\|_{L^p(Q_\tau)}+\|u_0\|_{W^{2-\frac{2}{p\beta},p}(\T)}).
\]
Notice that since $p>d+2/\beta$ we have $1/p<\beta/(d+2)<\beta$ for all $\beta\in(0,1)$, so that the initial trace is well-defined. We show that we can choose $\tau\in(0,T]$ sufficiently small so that $\|w\|_{\Vb_p^{2}(Q_\tau)}\leq a$. By \cite[Lemma 2.4]{CGM}
\[
\|H(x,Dz)\|_{L^p(Q_\tau)}\leq C_1\tau^{\frac{1}{2p}}\|H(x,Dz)\|_{L^{2p}(Q_{\tau})}\leq C_2\tau^{\frac{1}{2p}}\|Dz\|^{\gamma}_{\infty;Q_{\tau}}.
\]
Moreover, by the parabolic embeddings in Proposition \ref{emb}, we have
\[
\|Dz\|_{\infty;Q_\tau}\leq C_3(\|z\|_{\Vb_p^{2}(Q_\tau)}+\|u_0\|_{W^{2-\frac{2}{p\beta},p}(\T)})\ ,
\]
which both give
\[
\|H(x,Dz)\|_{L^p(Q_\tau)}\leq C_4\tau^{\frac{1}{2p}}(\|z\|_{\Vb_p^{2}(Q_\tau)}^\gamma +\|u_0\|_{W^{2-\frac{2}{p\beta},p}(\T)}^\gamma)\ .
\]

Then we are in position to show that $J$ maps $\mathcal{S}_a$ into itself. Indeed
\begin{multline*}
\|w\|_{\Vb_p^2(Q_\tau)}\leq C\left\{\|V\|_{L^p(Q_\tau)}+C_4\tau^{1/2p}(\|z\|_{\Vb_p^2(Q_\tau)}^\gamma+\|u_0\|_{W^{2-2/p\beta,p}(\T)}^\gamma)+\|u_0\|_{W^{2-2/p\beta,p}(\T)}\right\}\\
\leq C(\|V\|_{L^p(Q_\tau)}+\tau^{1/2p}\|z\|_{\Vb_p^2(Q_\tau)}^\gamma+\max\{\|u_0\|_{W^{2-2/p\beta,p}(\T)}^\gamma,\|u_0\|_{W^{2-2/p\beta,p}(\T)}\}).
\end{multline*}
We take 
\[
a\geq C(\max\{\|u_0\|_{W^{2-2/p\beta,p}(\T)}^\gamma,\|u_0\|_{W^{2-2/p\beta,p}(\T)}\}+\|V\|_{L^p(Q_\tau)})+1.
\]
Then, it follows that
\[
\|w\|_{\Vb_p^2(Q_\tau)}\leq C\tau^{1/2p}\|z\|_{\Vb_p^2(Q_\tau)}^\gamma+a-1.
\]
By taking $\tau$ sufficiently small we finally get $\|w\|_{\Vb_p^2(Q_\tau)}\leq a$.\\
To prove that $J$ is a contraction, one has to argue as above, exploiting also the fact that for bounded $z\in \Vb_p^{2}$, $p>d+\frac{2}{\beta}$, then $Dz$ is bounded in $L^{\infty}(Q_\tau)$. Using \cite[Lemma A.6]{DK} and using that $(z_1-z_2)(0)=0$ we have
\[
\|z_1-z_2\|_{L^p(Q_{\tau})}\leq C\tau^{\beta}\|\partial_{(0,t]}^{\beta} z\|_{L^p(Q_\tau)}
\]
for some positive $C$ depending merely on $\beta$ and $p$
%
and hence
\begin{multline*}
\norm{H(x,Dz_1)-H(x,Dz_2)}_{L^{p}(Q_\tau)}\leq C_1\norm{Dz_1-Dz_2}_{L^p(Q_\tau)}\\
\leq C_2\left(\epsilon\|D^2(z_1-z_2)\|_{L^p(Q_\tau)}+\epsilon^{-1}\|z_1-z_2\|_{L^p(Q_\tau)}\right)\\
\leq C_2\epsilon\|D^2(z_1-z_2)\|_{L^p(Q_\tau)}+C_3\epsilon^{-1}\tau^{\beta}\norm{\partial_{(0,t]}^{\beta}(z_1-z_2)}_{L^p(Q_\tau)},
\end{multline*}
for some positive constants $C_1,C_2,C_3$, where the second inequality is achieved via interpolation.

Then, by first taking $\epsilon$ small enough so that 
\[
C_1\epsilon\|D^2(z_1-z_2)\|_{L^p(Q_\tau)}\leq C_1\epsilon\|z_1-z_2\|_{\Vb_p^2(Q_{\tau})}\leq \frac14\|z_1-z_2\|_{\Vb_p^2(Q_{\tau})}
\]
and then choosing $\tau$ sufficiently small so that
\[
C_2\epsilon^{-1}\tau^{\beta-1/p}\norm{\partial_{(0,t]}^{\beta}(z_1-z_2)}_{L^p(Q_\tau)}\leq C_2\epsilon^{-1}\tau^{\beta}\|z_1-z_2\|_{\Vb_p^2(Q_{\tau})}\leq \frac14\|z_1-z_2\|_{\Vb_p^2(Q_{\tau})}\ ,
\]
we conclude that
\[
\|z_1-z_2\|_{\Vb_p^2(Q_{\tau})}\leq\frac12\|z_1-z_2\|_{\Vb_p^2(Q_{\tau})},
\]
ensuring the existence and uniqueness of a fixed point $Jz=z$ and hence a solution in the interval $[0,\tau]$. Moreover,
by Proposition \ref{emb} we have $z\in C([0,\tau];W^{2-\frac{2}{p\beta},p}(\T))$ and hence $|Dz|\in C([0,\tau];C^{\frac{\gamma}{\beta}}(\T))$. In view of the results in Section \ref{sec;schauder} (specifically  Theorem \ref{Schauder} or Theorem \ref{Regularity}) applied to the equation
\[
\partial_{(0,t]}^{\beta} u-\Delta u=f(x,t)
\]
with $f(x,t)=  V(x,t)-H(x,Dz) \in C([0,\tau];C^{\frac{\gamma}{\beta}}(\T))$,    we have $u\in C([0,\tau];C^{2+\frac{\gamma}{\beta}}(\T))$. Then, a bootstrap argument allows to conclude that $u\in C([0,\tau];C^{4+\frac{\gamma}{\beta}}(\T))$ using the regularity of $V$.\\

\par\smallskip
\textit{Step 2. Continuation of the solution on a larger interval $[0,\tau+\omega]$}. Let $w:\T\times[0,\tau]\to\R$ be the local classical solution to the fractional Hamilton-Jacobi equation with initial data $u_0$ in the interval $[0,\tau]$ obtained in the previous step. Define the set  
\[
\widetilde{\mathcal{S}}_a:=\left\{u\in \Vb_p^{2}(Q_{\tau+\omega}):u(t)=w(t)\ ,t\in[0,\tau]\ ,\|u\|_{\Vb_p^{2}(Q_{\tau+\omega})}\leq a\ ,p>d+\frac{2}{\beta}\right\}.
\]
Fix $z\in \widetilde{\mathcal{S}}_a$ and let $v=Jz$ be the solution to
\begin{equation*} 
\begin{cases}
\partial_{(0,t]}^{\beta} v-\Delta v=V-H(x,Dz)&\text{ in }Q_{\tau+\omega}\ ,\\
v(x,0)=u_0(x)&\text{ in }\T.
\end{cases}
\end{equation*}
Then, as above, for $\omega$ sufficiently small one proves that $\|Jz\|_{\Vb_p^{2}(Q_{\tau+\omega})}\leq a$. Indeed, if $z\in \widetilde{\mathcal{S}}_a$, then for all $t\in[0,\tau]$ we have $v(t)=Jz=w(t)$. For $t\in [\tau,\tau+\omega]$ using $L^p$ maximal regularity as in Step 1 one obtains that $\|v\|_{\Vb_p^{2}(Q_{\tau+\omega})}\leq a$. To prove that $J$ is a contraction, one can argue exactly in the same manner as in Step 1 to find
\[
\|Jz_1-Jz_2\|_{\Vb_p^{2}(Q_{\tau+\omega})}\leq\frac12\|z_1-z_2\|_{\Vb_p^{2}(Q_{\tau+\omega})}
\]
for every $z_1,z_2\in \widetilde{\mathcal{S}}_a$. Then, by the contraction mapping principle, one finds a unique fixed point $\bar{v}\in \widetilde{\mathcal{S}}_a$ which, by bootstrapping belongs to 
$C([0,\tau];C^{4+\frac{\gamma}{\beta}}(\T))$.
Repeating the same argument, one finds a maximal interval of existence $[0,\bar{T})$, where
\[
\bar{T}:=\sup\{\tau\in(0,T):\eqref{i;1} \text{ has a unique solution in $\Vb_p^2(Q_\tau)$}\}\ .
\]
\textit{Step 3. Global existence}. In this step, we restrict to $\beta \in (\frac 1 2,1)$ and we  use the a priori gradient bounds coming from Theorem \ref{semicest} to   show that $\|u\|_{\Vb_p^2(Q_\tau)}$ remains bounded as $\tau$ approaches to $\bar{T}$.\\
 Let $\tau\in (0,\bar{T})$ and $u\in \Vb_p^2(Q_\tau)$ be the unique solution to \eqref{i;1}, which, by bootstrapping, is a posteriori a classical solution. By Theorem \ref{semicest} with $\sigma=1$, we have
\[
\|Du\|_{L^p(Q_\tau)}\leq C,
\]
for every $p\ge 1 $, for some $C$ independently of $\tau$. Then, the Hamilton-Jacobi equation can  be written as
\[
\partial_{(0,t]}^\beta u-\Delta u=V(x,t)-H(x,Du)
\]
with initial data $u(x,0)=u_0$. Consider then the solution to the linear problem
\[
\begin{cases}
\partial_{(0,t]}^{\beta} v-\Delta v=V(x,t)&\text{ in }Q_\tau\ ,\\
v(x,0)=u_0(x)&\text{ in }\T.
\end{cases}
\]
By Theorem \ref{thm;maxreg}, the previous problem admits a unique solution $v\in \Vb_p^2(Q_\tau)$. In particular, one has
\[
\|u-v\|_{\Vb_p^2(Q_\tau)}\leq C_1\|H(x,Du)\|_{L^p(Q_\tau)}\leq  C_2(\|Du\|_{L^{p\gamma}(Q_\tau)}^\gamma+1)\leq C_3
\]
for some $C_3$ independent of $\tau$ by exploiting the gradient bound in Theorem \ref{semicest}. Since $\|v\|_{\Vb_p^2(Q_\tau)}$ stays bounded for $\tau\nearrow\bar{T}$, the same is true for $\|v\|_{\Vb_p^2(Q_\tau)}$, yielding thus the global existence.\\
\end{proof}
 Some final comments on the results are in order
\begin{rem}\label{stclassical}
Note that a space-time H\"older regularity result for the time-fractional Hamilton-Jacobi equation \eqref{i;1} can be obtained using Theorem \ref{Regularity} instead of Theorem \ref{Schauder}. More precisely, one first show, as in the above Step 1, that $z\in \Vb_p^2(Q)$. First, since $\partial_i$, $i=1,...,d$, is a bounded linear operator from $\Vb_p^2(Q)$ to $H_p^{\beta/2}(0,\tau;L^p(\T))\cap L^p(0,\tau;W^{1,p}(\T))$ (see e.g. \cite[Proposition 2.4]{CK} for the whole space case, the periodic case can be treated using transference arguments from $\R^d$ to $\T$ as in \cite{CG1}), it follows that $|Dz|\in H_p^{\beta/2}(0,\tau;L^p(\T))\cap L^p(0,\tau;W^{1,p}(\T))$. This immediately implies by the Mixed Derivative Theorem that 
\[
|Dz|\in H_p^{\frac\beta2}(0,\tau;L^p(\T))\cap L^p(0,\tau;W^{1,p}(\T))\hookrightarrow H_p^{\frac{\beta\zeta}{2}}(0,\tau;H_p^{1-\zeta}(\T))
\] 
for all $\zeta\in[0,1]$, whereas, owing to Remark \ref{spacetimeholder}, it is possible to conclude that $|Dz|$ belong to $C^{\gamma_1}(C^{\gamma_2})$ for suitable $\gamma_1,\gamma_2\in (0,1)$ whenever $p>d+2/\beta$. This finally gives that $H(x,Du)\in C^{\gamma_1}(C^{\gamma_2})$ and Theorem \ref{Regularity} eventually yields the space-time H\"older regularity result.\\
The latter Sobolev embedding is reminiscent of those for the parabolic class $\Hb_p^{1}$: we refer e.g. to \cite[Theorem A.7]{DK2}, where similar embeddings are proved in the time-fractional setting, but not in the periodic case (see also \cite[Proposition 2.2]{CT} for its local counterpart on the torus and \cite[Appendix A]{MPR} on the whole space).\\
Furthermore, note that the restriction $\beta\in(1/2,1)$ is necessary and a consequence of Theorem \ref{Schauder}. Heuristically, this is due to the fact that under the regime $\beta\in(1/2,1)$ the time-fractional parabolic operator behaves like the heat operator and this allows to use standard decay estimates of the heat semigroup.
\end{rem}

\begin{rem}
The restriction on $p>d+2/\beta$ in Theorem \ref{existenceHamilton-Jacobi}, when applying the contraction mapping procedure, is consistent with Lipschitz regularity results for time-fractional heat equations with $L^p$ right-hand side. In particular, let us consider the problem
\[
\begin{cases}
\partial_{(0,t]}^{\beta} v-\Delta v=f(x,t)&\text{ in }Q_T\ ,\\
v(x,0)=v_0(x)&\text{ in }\T.
\end{cases}
\]
with $v_0\in W^{2-\frac{2}{p\beta},p}(\T)$ and $f\in L^p(Q_T)$. Maximal $L^p$-regularity result and Sobolev embedding theorems yield $v\in \Vb_p^2(Q_T)$ and
\[
\Vb_p^2(Q_T)\hookrightarrow C(0,T;W^{2-\frac{2}{p\beta},p}(\T))\ .
\]
Then, one observes that if $v\in C(0,T;W^{2-\frac{2}{p\beta},p}(\T))$, then $|Dv|\in C(0,T;W^{1-\frac{2}{p\beta},p}(\T))$, and one finally concludes using that $W^{1-\frac{2}{p\beta},p}(\T)$ is embedded at least in $C(Q_T)$ whenever $(1-\frac{2}{p\beta})p>d$ by Lemma \ref{inclstatW}-(ii), i.e. $p>d+2/\beta$.
\end{rem}


\small

\begin{thebibliography}{10}

\bibitem{AV}
N.~Abatangelo and E.~Valdinoci.
\newblock Getting acquainted with the fractional {L}aplacian.
\newblock In {\em Contemporary research in elliptic {PDE}s and related topics},
  volume~33 of {\em Springer INdAM Ser.}, pages 1--105. Springer, Cham, 2019.

\bibitem{ACV}
M.~Allen, L.~Caffarelli, and A.~Vasseur.
\newblock A parabolic problem with a fractional time derivative.
\newblock {\em Arch. Ration. Mech. Anal.}, 221(2):603--630, 2016.

\bibitem{Bagby}
R.~J. Bagby.
\newblock Lebesgue spaces of parabolic potentials.
\newblock {\em Illinois J. Math.}, 15:610--634, 1971.

\bibitem{BL}
J.~Bergh and J.~L\"ofstr\"om.
\newblock {\em Interpolation spaces. {A}n introduction}.
\newblock Springer-Verlag, Berlin-New York, 1976.
\newblock Grundlehren der Mathematischen Wissenschaften, No. 223.

\bibitem{BKRS}
V.~I. Bogachev, N.~V. Krylov, M.~R\"ockner, and S.~V. Shaposhnikov.
\newblock {\em Fokker-{P}lanck-{K}olmogorov equations}, volume 207 of {\em
  Mathematical Surveys and Monographs}.
\newblock American Mathematical Society, Providence, RI, 2015.

\bibitem{CF}
L.~Caffarelli and A.~Figalli.
\newblock Regularity of solutions to the parabolic fractional obstacle problem.
\newblock {\em J. Reine Angew. Math.}, 680:191--233, 2013.

\bibitem{cdm}
F.~Camilli and R.~De~Maio.
\newblock A time-fractional mean field game.
\newblock {\em Adv. Differential Equations}, 24(9-10):531--554, 2019.

\bibitem{CLLP}
P.~Cardaliaguet, J.-M. Lasry, P.-L. Lions, and A.~Porretta.
\newblock Long time average of mean field games with a nonlocal coupling.
\newblock {\em SIAM J. Control Optim.}, 51(5):3558--3591, 2013.

\bibitem{CK}
T.~Chang and K.~Kang.
\newblock Estimates of anisotropic {S}obolev spaces with mixed norms for the
  {S}tokes system in a half-space.
\newblock {\em Ann. Univ. Ferrara Sez. VII Sci. Mat.}, 64(1):47--82, 2018.

\bibitem{CL}
T.~Chang and K.~Lee.
\newblock On a stochastic partial differential equation with a fractional
  {L}aplacian operator.
\newblock {\em Stochastic Process. Appl.}, 122(9):3288--3311, 2012.

\bibitem{CGM}
M.~Cirant, R.~Gianni, and P.~Mannucci.
\newblock Short-time existence for a backward-forward parabolic system arising
  from mean-field games.
\newblock arXiv:1806.08138, to appear in Dynamic Games and Applications, 2018.

\bibitem{CG2}
M.~Cirant and A.~Goffi.
\newblock Lipschitz regularity for viscous {H}amilton-{J}acobi equations with
  {$L^p$} terms.
\newblock arXiv:1812.03706, to appear on Ann. Inst. H. Poincar\'e Anal. Non
  Lin\'eaire, 2018.

\bibitem{CG1}
M.~Cirant and A.~Goffi.
\newblock On the existence and uniqueness of solutions to time-dependent
  fractional {MFG}.
\newblock {\em SIAM J. Math. Anal.}, 51(2):913--954, 2019.

\bibitem{CT}
M.~Cirant and D.~Tonon.
\newblock Time-dependent focusing mean-field games: the sub-critical case.
\newblock {\em J. Dynam. Differential Equations}, 31(1):49--79, 2019.

\bibitem{ClementTAMS}
P.~Cl\'{e}ment, G.~Gripenberg, and S.-O. Londen.
\newblock Schauder estimates for equations with fractional derivatives.
\newblock {\em Trans. Amer. Math. Soc.}, 352(5):2239--2260, 2000.

\bibitem{CLS}
P.~Cl\'{e}ment, S.-O. Londen, and G.~Simonett.
\newblock Quasilinear evolutionary equations and continuous interpolation
  spaces.
\newblock {\em J. Differential Equations}, 196(2):418--447, 2004.

\bibitem{Andrade}
B.~de~Andrade, A.~N. Carvalho, P.~M. Carvalho-Neto, and P.~Mar\'{i}n-Rubio.
\newblock Semilinear fractional differential equations: global solutions,
  critical nonlinearities and comparison results.
\newblock {\em Topol. Methods Nonlinear Anal.}, 45(2):439--467, 2015.

\bibitem{Neto}
P.~M. de~Carvalho-Neto and G.~Planas.
\newblock Mild solutions to the time fractional {N}avier-{S}tokes equations in
  {$\mathbb{R}^N$}.
\newblock {\em J. Differential Equations}, 259(7):2948--2980, 2015.

\bibitem{DK}
H.~Dong and D.~Kim.
\newblock {$L_p$}-estimates for time fractional parabolic equations with
  coefficients measurable in time.
\newblock {\em Adv. Math.}, 345:289--345, 2019.

\bibitem{DK2}
H.~Dong and D.~Kim.
\newblock {$L_p$}-estimates for time fractional parabolic equations in
  divergence form with measurable coefficients.
\newblock {\em J. Funct. Anal.}, 278(3):108338, 66, 2020.

\bibitem{Evans}
L.~C. Evans.
\newblock Adjoint and compensated compactness methods for {H}amilton-{J}acobi
  {PDE}.
\newblock {\em Arch. Ration. Mech. Anal.}, 197(3):1053--1088, 2010.

\bibitem{FRRO}
X.~Fern\'andez-Real and X.~Ros-Oton.
\newblock Regularity theory for general stable operators: parabolic equations.
\newblock {\em J. Funct. Anal.}, 272(10):4165--4221, 2017.

\bibitem{giganamba}
Y.~Giga and T.~Namba.
\newblock Well-posedness of {H}amilton-{J}acobi equations with {C}aputo's time
  fractional derivative.
\newblock {\em Comm. Partial Differential Equations}, 42(7):1088--1120, 2017.

\bibitem{gomesbook}
D.~A. Gomes, E.~A. Pimentel, and V.~Voskanyan.
\newblock {\em Regularity theory for mean-field game systems}.
\newblock Springer Briefs in Mathematics. Springer, [Cham], 2016.

\bibitem{GopalaRaoTAMS}
V.~R. Gopala~Rao.
\newblock Parabolic function spaces with mixed norm.
\newblock {\em Trans. Amer. Math. Soc.}, 246:451--461, 1978.

\bibitem{Guidetti}
D.~Guidetti.
\newblock On maximal regularity for abstract parabolic problems with fractional
  time derivative.
\newblock {\em Mediterr. J. Math.}, 16(2):Art. 40, 26, 2019.

\bibitem{KZ}
J.~Kemppainen and R.~Zacher.
\newblock Long-time behavior of non-local in time {F}okker-{P}lanck equations
  via the entropy method.
\newblock {\em Math. Models Methods Appl. Sci.}, 29(2):209--235, 2019.

\bibitem{KKL}
I.~Kim, K.-H. Kim, and S.~Lim.
\newblock An {$L_q(L_p)$}-theory for the time fractional evolution equations
  with variable coefficients.
\newblock {\em Adv. Math.}, 306:123--176, 2017.

\bibitem{kolokver}
V.~N. Kolokoltsov and M.~A. Veretennikova.
\newblock Well-posedness and regularity of the {C}auchy problem for nonlinear
  fractional in time and space equations.
\newblock {\em Fract. Differ. Calc.}, 4(1):1--30, 2014.

\bibitem{KrylovbookSPDE}
N.~V. Krylov.
\newblock An analytic approach to {SPDE}s.
\newblock In {\em Stochastic partial differential equations: six perspectives},
  volume~64 of {\em Math. Surveys Monogr.}, pages 185--242. Amer. Math. Soc.,
  Providence, RI, 1999.

\bibitem{KrylovJFA}
N.~V. Krylov.
\newblock Some properties of traces for stochastic and deterministic parabolic
  weighted {S}obolev spaces.
\newblock {\em J. Funct. Anal.}, 183(1):1--41, 2001.

\bibitem{LSU}
O.~A. Ladyzenskaja, V.~A. Solonnikov, and N.~N. Ural'ceva.
\newblock {\em Linear and quasilinear equations of parabolic type}.
\newblock Translated from the Russian by S. Smith. Translations of Mathematical
  Monographs, Vol. 23. American Mathematical Society, Providence, R.I., 1968.

\bibitem{tly}
O.~Ley, E.~Topp, and M.~Yangari.
\newblock Some results for the large time behavior of {H}amilton-{J}acobi
  equations with {C}aputo {T}ime derivative.
\newblock arXiv:1906.06625, 2019.

\bibitem{LiLiuODE}
L.~Li and J.-G. Liu.
\newblock A generalized definition of {C}aputo derivatives and its application
  to fractional {ODE}s.
\newblock {\em SIAM J. Math. Anal.}, 50(3):2867--2900, 2018.

\bibitem{LinLiu}
L.~Li and J.-G. Liu.
\newblock Some compactness criteria for weak solutions of time fractional
  {PDE}s.
\newblock {\em SIAM J. Math. Anal.}, 50(4):3963--3995, 2018.

\bibitem{KellerSegel}
L.~Li, J.-G. Liu, and L.~Wang.
\newblock Cauchy problems for {K}eller-{S}egel type time-space fractional
  diffusion equation.
\newblock {\em J. Differential Equations}, 265(3):1044--1096, 2018.

\bibitem{Luchko}
Y.~Luchko.
\newblock Maximum principle for the generalized time-fractional diffusion
  equation.
\newblock {\em J. Math. Anal. Appl.}, 351(1):218--223, 2009.

\bibitem{LunardiNote}
A.~Lunardi.
\newblock How to use interpolation in pde's.
\newblock On-line lecture notes, available at
  http://people.dmi.unipr.it/alessandra.lunardi/.

\bibitem{LunardiBook}
A.~Lunardi.
\newblock {\em Analytic semigroups and optimal regularity in parabolic
  problems}.
\newblock Modern Birkh\"auser Classics. Birkh\"auser/Springer Basel AG, Basel,
  1995.
\newblock [2013 reprint of the 1995 original] [MR1329547].

\bibitem{LunardiSNS}
A.~Lunardi.
\newblock {\em Interpolation theory}, volume~16 of {\em Appunti. Scuola Normale
  Superiore di Pisa (Nuova Serie) [Lecture Notes. Scuola Normale Superiore di
  Pisa (New Series)]}.
\newblock Edizioni della Normale, Pisa, 2018.
\newblock Third edition [of MR2523200].

\bibitem{MPR}
G.~Metafune, D.~Pallara, and A.~Rhandi.
\newblock Global properties of transition probabilities of singular diffusions.
\newblock {\em Teor. Veroyatn. Primen.}, 54(1):116--148, 2009.

\bibitem{MS}
M.~Meyries and R.~Schnaubelt.
\newblock Interpolation, embeddings and traces of anisotropic fractional
  {S}obolev spaces with temporal weights.
\newblock {\em J. Funct. Anal.}, 262(3):1200--1229, 2012.

\bibitem{MV}
M.~Meyries and M.~C. Veraar.
\newblock Traces and embeddings of anisotropic function spaces.
\newblock {\em Math. Ann.}, 360(3-4):571--606, 2014.

\bibitem{Porretta}
A.~Porretta.
\newblock Weak solutions to {F}okker-{P}lanck equations and mean field games.
\newblock {\em Arch. Ration. Mech. Anal.}, 216(1):1--62, 2015.

\bibitem{Vergara}
J.~C. Pozo and V.~Vergara.
\newblock Fundamental solutions and decay of fully non-local problems.
\newblock {\em Discrete Contin. Dyn. Syst.}, 39(1):639--666, 2019.

\bibitem{Prussmaximal}
J.~Pr\"{u}ss.
\newblock Maximal regularity of linear vector-valued parabolic {V}olterra
  equations.
\newblock {\em J. Integral Equations Appl.}, 3(1):63--83, 1991.

\bibitem{PrussBook}
J.~Pr\"{u}ss.
\newblock {\em Evolutionary integral equations and applications}.
\newblock Modern Birkh\"{a}user Classics. Birkh\"{a}user/Springer Basel AG,
  Basel, 1993.
\newblock [2012] reprint of the 1993 edition.

\bibitem{PrussSimonett}
J.~Pr\"{u}ss and G.~Simonett.
\newblock {\em Moving interfaces and quasilinear parabolic evolution
  equations}, volume 105 of {\em Monographs in Mathematics}.
\newblock Birkh\"{a}user/Springer, [Cham], 2016.

\bibitem{Rabier}
P.~J. Rabier.
\newblock Vector-valued {M}orrey's embedding theorem and {H}\"older continuity
  in parabolic problems.
\newblock {\em Electron. J. Differential Equations}, pages No. 10, 10, 2011.

\bibitem{Sinestrari}
E.~Sinestrari.
\newblock On the abstract {C}auchy problem of parabolic type in spaces of
  continuous functions.
\newblock {\em J. Math. Anal. Appl.}, 107(1):16--66, 1985.

\bibitem{TangDCDS}
Q.~Tang.
\newblock On an optimal control problem of time-fractional advection-diffusion
  equation.
\newblock {\em Discrete Contin. Dyn. Syst. Ser. B}, 25(2):761--779, 2020.

\bibitem{TaylorNote}
M.~Taylor.
\newblock Remarks on fractional diffusion equations.
\newblock On-line lecture notes, available at
  \url{http://mtaylor.web.unc.edu/files/2018/04/fdif.pdf}.

\bibitem{toppyangari}
E.~Topp and M.~Yangari.
\newblock Existence and uniqueness for parabolic problems with {C}aputo time
  derivative.
\newblock {\em J. Differential Equations}, 262(12):6018--6046, 2017.

\bibitem{trbookinterpolation}
H.~Triebel.
\newblock {\em Interpolation theory, function spaces, differential operators}.
\newblock Johann Ambrosius Barth, Heidelberg, second edition, 1995.

\bibitem{WCX}
R.-N. Wang, D.-H. Chen, and T.-J. Xiao.
\newblock Abstract fractional {C}auchy problems with almost sectorial
  operators.
\newblock {\em J. Differential Equations}, 252(1):202--235, 2012.

\bibitem{TesiZacher}
R.~Zacher.
\newblock {\em Quasilinear parabolic problems with nonlinear boundary
  conditions}.
\newblock PhD thesis.
\newblock 2003.

\bibitem{ZacherIBVP}
R.~Zacher.
\newblock Maximal regularity of type {$L_p$} for abstract parabolic {V}olterra
  equations.
\newblock {\em J. Evol. Equ.}, 5(1):79--103, 2005.

\bibitem{Zacherabstract}
R.~Zacher.
\newblock Weak solutions of abstract evolutionary integro-differential
  equations in {H}ilbert spaces.
\newblock {\em Funkcial. Ekvac.}, 52(1):1--18, 2009.

\bibitem{ZacherGlobal}
R.~Zacher.
\newblock Global strong solvability of a quasilinear subdiffusion problem.
\newblock {\em J. Evol. Equ.}, 12(4):813--831, 2012.

\bibitem{Zmax}
R.~Zacher.
\newblock Time fractional diffusion equations: solution concepts, regularity,
  and long-time behavior.
\newblock In {\em Handbook of fractional calculus with applications. {V}ol. 2},
  pages 159--179. De Gruyter, Berlin, 2019.

\end{thebibliography}

\medskip
\begin{flushright}
		\noindent \verb"camilli@sbai.uniroma1.it"\\
		SBAI, Universit\`{a} di Roma ``La Sapienza"\\
		via A.Scarpa 14, 00161 Roma (Italy)
		
\end{flushright}
\begin{flushright}

\noindent \verb"alessandro.goffi@gssi.it"\\
Gran Sasso Science Institute\\
viale Francesco Crispi 7, 67100 L'Aquila (Italy)
\end{flushright}

\end{document}